\documentclass[11pt]{article}

\usepackage{float}
\usepackage{amsfonts,yhmath}
\usepackage{amsmath}
\usepackage{amsthm}
\usepackage{amssymb}
\usepackage{enumitem}
\usepackage{color}
\usepackage{url}
\usepackage{geometry}
\usepackage{bm}
\usepackage{graphicx}
\usepackage{subcaption}

\usepackage{tikz}
\tikzstyle{vertex}=[circle, draw, inner sep=2pt, fill=white]

\newtheorem{thm}{Theorem}[section]

\newtheorem{lem}[thm]{Lemma}
\newtheorem{prop}[thm]{Proposition}

\theoremstyle{definition}
\newtheorem{defn}[thm]{Definition}

\newtheorem{rem}[thm]{Remark}

\numberwithin{equation}{section}
\numberwithin{thm}{section}

\newcommand{\R}{\mathbb{R}}

\newcommand{\PP}{\mathbb{P}}

\newcommand{{\EE}}{\mathbb{E}}

\newcommand{\dbra}[1]{[\kern-0.15em[ #1 ]\kern-0.15em]}
\newcommand{\dbraco}[1]{[\kern-0.15em[ #1 [\kern-0.15em[}
\newcommand{\dbraoc}[1]{]\kern-0.15em] #1 ]\kern-0.15em]}
\newcommand{\dbraoo}[1]{]\kern-0.15em] #1 [\kern-0.15em[}

\newcommand{\Ccal}{\mathcal{C}}

\newcommand{\Fcal}{\mathcal{F}}
\newcommand{\Gcal}{\mathcal{G}}

\newcommand{\Lcal}{\mathcal{L}}
\newcommand{\Ncal}{\mathcal{N}}

\newcommand{\Pcal}{\mathcal{P}}

\newcommand{\FF}{\mathbb{F}}

\newcommand{\HH}{\mathbb{H}}

\newcommand{\RR}{\mathbb{R}}
\newcommand{\NN}{\mathbb{N}}

\def\sw{{\overline{\omega}}}

\def\balpha{\boldsymbol{\alpha}}
\def\bbeta{\boldsymbol{\beta}}
\def\bgamma{\boldsymbol{\gamma}}
\def\eeta{\boldsymbol{\eta}}

\def\bmu{\boldsymbol{\mu}}
\def\bnu{\boldsymbol{\nu}}

\def\bzeta{\boldsymbol{\zeta}}
\def\bM{\boldsymbol{M}}
\def\bV{\boldsymbol{V}}
\def\bX{\boldsymbol{X}}
\def\bS{\boldsymbol{S}}
\def\bY{\boldsymbol{Y}}
\def\bZ{\boldsymbol{Z}}
\def\bW{\boldsymbol{W}}

\newcommand{\cF}{\mathcal F}
\newcommand{\cL}{\mathcal L}
\newcommand{\cP}{\mathcal P}

\makeatletter
\newcommand{\LeftEqNo}{\let\veqno\@@leqno}
\makeatother

\allowdisplaybreaks

\begin{document}

\title{Extended Mean Field Control Problems:  stochastic maximum principle and  transport perspective}
\author{
	Beatrice Acciaio\thanks{Department of Statistics, London School of Economics}
	\and
	Julio Backhoff-Veraguas\thanks{ Institute of Statistics and Mathematical Methods in Economics ,Vienna University of Technology}
	\and
	Ren\'{e} Carmona\thanks{Operations Research and Financial Engineering, Princeton University, Partially supported by NSF \# DMS-1716673 and ARO \# W911NF-17-1-0578} 
}

\date{\today}

\maketitle

\begin{abstract}
We study Mean Field stochastic control problems where the cost
function and the state dynamics depend upon the
joint distribution of the controlled state and the control process.
We prove suitable versions of the Pontryagin stochastic maximum principle, both in necessary and
in sufficient form, which extend the known conditions to this general framework. 
Furthermore, we suggest a variational approach to study
a weak formulation of these control problems. We show a natural connection between this weak formulation and optimal transport on path space, which inspires a novel discretization scheme.
\end{abstract}

\section{Introduction}
\label{sect:introduction}

The control of stochastic differential equations of Mean Field type, also known as McKean-Vlasov control, did not get much attention before the theory of Mean Field Games became a popular subject of investigation. Indeed the two topics are intimately related through the asymptotic theory of mean field stochastic systems known as \emph{propagation of chaos}. See for example \cite{CDL} for an early discussion of the similarities and the differences of the two problems. Among the earliest works on this new form of control problem relevant to the spirit of the analysis conducted  in this paper are \cite{buckdahn2009mean,buckdahn2011general,andersson2011maximum,meyer2012mean,BFY,CD_AP}. Here, we follow the approach introduced and developed in \cite{CD_AP}.
The reader is referred to Chapters 3, 4, and 6 of \cite{CarmonaDelarue_book_vol_I} for a general overview of these problems and an extensive historical perspective. Still, despite the section devoted to the so-called extended mean field games in \cite[Ch.\ 4]{CarmonaDelarue_book_vol_I}, most of these contributions are limited to Mean Field interactions entering the models through the statistical distribution of the state of the system alone.  The goal of the present article is to investigate  the control of stochastic dynamics depending upon the joint distribution of the controlled state and the control process. We refer to such problems as \emph{extended Mean Field control problems}.

Our first contribution is to prove an appropriate form of the Pontryagin stochastic maximum principle, in necessary and in sufficient form, for extended Mean Field control problems. The main driver behind this search for an extension of existing tools, is the importance of many practical applications which naturally fit within the class of models for which the interactions are not only through the distribution of the state of the system, but also through the distribution of the controls used by the controller. The analysis of extended Mean Field control problems had been restricted so far to the Linear Quadratic (LQ) case; see e.g.\ \cite{BP,PW,yong2013linear,graber2016linear}. To the best of our knowledge, the recent work \cite{PW} is the only one where more general models are considered. In that article, however, the authors restrict the analysis to closed-loop feedback controls, leading to a deterministic reformulation of the problem, which is dealt with PDE techniques. In the present paper, we study the extended Mean Field control problem without any restrictions, deriving through the probabilistic approach a Pontryagin maximum principle. 

We apply our optimality conditions for particular classes of models, where our analysis can be pushed further. In the case of scalar interactions, we derive a more explicit form of the optimality condition. The advantage here is that the
analysis can be conducted with a form of classical differential calculus, without the use of the notion of L-differentiability. The announced work \cite{G} studies an application of such class of models in electricity markets. As a special case of scalar interaction, we study an optimal liquidation model, which we are able to solve explicitly. Finally, we consider the case of LQ  models for which we easily derive explicit solutions which can be computed numerically. The results in the LQ setting are compatible with existing results in the literature.

Another contribution of the present article, is the variational study of a weak formulation of the extended Mean Field control problem. Weak formulations have already been studied in the literature, without non-linear dependence in the law of the control, as in \cite[Ch.\ 6]{CarmonaDelarue_book_vol_I} and \cite{La17}. In this framework, we derive an analogue of the Pontryagin principle it the form of a {martingale optimality condition}. Similar statements have been derived in \cite{LZ,CL} under the name of Stochastic Euler-Lagrange condition for different kind of problems. Next, we derive a natural connection between the extended Mean Field control problem and an optimal transport problem on path space. The theory of optimal transport is known to provide a set of tools and results crucial to the understanding of mean field control and mean field games. We illustrate the usefulness of this connection, by building  a discretization scheme for extended Mean Field control based on transport-theoretic tools (as in \cite[Ch.\ 3.6]{Z} for the case without Mean Field terms), and show that this scheme converges monotonically to the value of {the original extended Mean Field control problem}. The explosion in activity regarding numerical optimal transport gives us reason to believe that such discretization schemes can be efficiently implemented; see e.g.\ \cite{cuturi2013sinkhorn,benamou2015iterative,2018-Peyre-computational-ot} for the static setting and\ \cite{pflug2009version,pflug2012distance,pflug2016multistage}  for the dynamic one. 
 
The paper is organized as follows. Section \ref{sect:gMKV} introduces the notations and basic underpinnings for extended mean field control. Section~\ref{sect:Pontryagin} provides the new form of the Pontryagin stochastic maximum principle alluded to earlier. 
In Section \ref{sect:examples} we study classes of models for which our optimality conditions lead to explicit solutions. Section \ref{sect:wMKV} analyses the weak formulation of the problem in connection with optimal transport.  Finally, we collect in the Appendix some technical proofs.

\section{Extended Mean Field Control Problems}
\label{sect:gMKV}

The goal of this short subsection is to set the stage for the statements and proofs of the stochastic maximum principle proven in Section \ref{sect:Pontryagin} below.

\vskip 2pt
Let  $f$, $b$, and $\sigma$ be measurable functions on $\RR^d\times\RR^k\times\cP_2(\RR^d\times\RR^k)$ with values in $\RR$, $\RR^d$, and $\RR^{d\times m}$ respectively, and $g$ be a real valued measurable function on $\RR^d\times\cP_2(\RR^d)$. 
Let  $(\Omega,\cF,\PP)$ be a probability space, $\cF_0\subset \cF$ be a sub sigma-algebra, and $\FF=(\cF_t)_{0\le t\le T}$ the filtration generated by $\cF_0$ and an $m$-dimensional Wiener process $\bW=(W_t)_{0\le t\le T}$. 
We denote by $\mathbb A$ the set of progressively measurable processes $\balpha=(\alpha_t)_{0\le t\le T}$ taking values in a given closed-convex set $A\subset\RR^k$, and satisfying the integrability condition $\EE\int_0^T|\alpha_t|^2dt<\infty$.

We consider the problem of minimizing
$$
J(\balpha)={\mathbb E}\Bigl[\int_0^T f(X_t,\alpha_t,\cL(X_t,\alpha_t)) dt +g\bigl(X_T,\cL(X_T)\bigr)\Bigr]
$$
over the set ${\mathbb A}$ of admissible control processes, under the {dynamic} constraint
\begin{equation}
\label{fo:state}
dX_t=b(X_t,\alpha_t,\cL(X_t,\alpha_t))dt + \sigma(X_t,\alpha_t,\cL(X_t,\alpha_t))dW_t,
\end{equation}
with $X_0$ a fixed $\cF_0$-measurable random variable. 

The symbol $\cL$ stands for the law of the given random element. We shall add mild regularity conditions for the coefficients $b$ and $\sigma$ so that a solution to equation \eqref{fo:state} always exists when $\balpha\in{\mathbb A}$.
For the sake of simplicity, we chose to use time independent coefficients, but all the results would be the same should $f$, $b$ and $\sigma$ depend upon $t$, since time can be included as an extra state in the vector $X$.

The novelty of the above control problem lies in the fact that the cost functional and the controlled SDE depend on the joint distribution of state and control. For this reason we call it \emph{extended Mean Field control problem}. In this generality, this problem has not been studied before. We mention the works \cite{PW,BP,graber2016linear,yong2013linear} for an analysis of particular cases, and different approaches.  

\subsection{Partial L-differentiability of Functions of Measures}
In this subsection, we explain how to compute partial L-derivatives of functions of joint probability laws (i.e. probability measures on product spaces).
We refer to Chapter 5 of \cite{CarmonaDelarue_book_vol_I} for details on the notion of L-differentiation of functions of measures.

Let $u:\RR^q\times \cP_2(\RR^d\times\RR^k) \to \RR$. We use the notation $\xi$ for a generic element of $\cP_2(\RR^d\times\RR^k)$, and $\mu\in\cP_2(\RR^d)$ and $\nu\in\cP_2(\RR^k)$ for its marginals. W denote here by $v$ a generic element of $\RR^q$.  

Let $(\tilde\Omega,\tilde\cF,\tilde\PP)$ be a probability space and $\tilde u$ a lifting of the function $u$. In other words:
$$
\tilde u :\RR^q\times L^2(\tilde\Omega,\tilde\cF,\tilde\PP;\RR^d\times\RR^k)\ni (v,\tilde X,\tilde\alpha)\mapsto \tilde u(v,\tilde X,\tilde\alpha)=u(v,\cL(\tilde X,\tilde\alpha)).
$$
We say that $u$ is L-differentiable at $(v,\xi)$ if there exists a pair $$(\tilde X,\tilde\alpha)\in L^2(\tilde\Omega,\tilde\cF,\tilde\PP;\RR^d\times\RR^k)\, \text{ with }\,\cL(\tilde X,\tilde\alpha)=\xi,$$ such that the lifted function $\tilde u$ is Fr\'echet differentiable at $(v,\tilde X,\tilde\alpha)$. When this is the case, it turns out that the Fr\'echet derivative depends only on the law $\xi$ and not on the specific pair $(\tilde X,\tilde\alpha)$ having distribution $\xi$; see \cite{Cardaliaguet} or Chapter 6 of \cite{CarmonaDelarue_book_vol_I} for details. The Fr\'echet derivative $[D\tilde u](v,\tilde X,\tilde\alpha)$ of the lifting function $\tilde u$ at $(v,\tilde X,\tilde\alpha)$ can be viewed as an element $D\tilde u(v,\tilde X,\tilde\alpha)$ of $\RR^q\times L^2(\tilde\Omega,\tilde\cF,\tilde\PP;\RR^d\times\RR^k)$, in the sense that
\[
[D\tilde u](v,\tilde X,\tilde\alpha)(\tilde Y)=\tilde\EE[D\tilde u(v,\tilde X,\tilde\alpha)\cdot \tilde Y],\quad \textrm{for all}\; \tilde Y\in \RR^q\times L^2(\tilde\Omega,\tilde\cF,\tilde\PP;\RR^d\times\RR^k).
\]
Since $\RR^q\times L^2(\tilde\Omega,\tilde\cF,\tilde\PP;\RR^d\times\RR^k)\cong \RR^q\times L^2(\tilde\Omega,\tilde\cF,\tilde\PP;\RR^d)\times L^2(\tilde\Omega,\tilde\cF,\tilde\PP;\RR^k)$, as in \cite{Cardaliaguet} the random variable $D\tilde u(v,\tilde X,\tilde\alpha)$ can be represented a.s.\ via the random vector
\[
D\tilde u (v,\tilde X,\tilde\alpha)=\left (\partial_v u(v,\cL(\tilde X,\tilde\alpha))(\,\tilde X,\tilde\alpha)\, ,\, \partial_\mu u(v,\cL(\tilde X,\tilde\alpha))(\tilde X,\tilde\alpha)\, ,\, \partial_\nu u(v,\cL(\tilde X,\tilde\alpha))(\tilde X,\tilde\alpha)\, \right) ,
\] 
for measurable functions $\partial_v u(\cdot,\cL(\tilde X,\tilde\alpha))(\cdot,\cdot)$, $\partial_\mu u(\cdot,\cL(\tilde X,\tilde\alpha))(\cdot,\cdot)$, $\partial_\nu u(\cdot,\cL(\tilde X,\tilde\alpha))(\cdot,\cdot)$, all of them defined on $\RR^q\times \RR^d\times\RR^k$ and valued respectively on $\RR^q$, $\RR^d$ and $\RR^k$. We call these functions the partial L-derivatives of $u$ at $(v,\cL(\tilde X,\tilde\alpha))$.

\section{Stochastic Maximum Principle}
\label{sect:Pontryagin}

Our goal is to prove a necessary and a sufficient condition for optimality in the extended class of problems considered in the paper. These are suitable extensions of the Pontryagin stochastic maximum principle conditions.
We define the Hamiltonian  $H$ by:
\begin{equation*}
\label{fo:hamiltonian}
\begin{split}
H(x,\alpha,\xi,y,z)&=b(x,\alpha,\xi)\cdot y +\sigma(x,\alpha,\xi)\cdot z + f(x,\alpha,\xi), 
\end{split}
\end{equation*}
for $(x,\alpha,\xi,y,z) \in \RR^d \times\RR^k\times \cP_{2}(\RR^d\times\RR^k)\times \RR^d\times\RR^{d\times m}$. Naturally, the dot notation for matrices refers to the trace inner product. We let $\HH^{0,n}$ stand for the collection of all $\RR^n$-valued progressively measurable processes on $[0,T]$, and denote by $\HH^{2,n}$ the collection of processes $Z$ in $\HH^{0,n}$ such that $\EE\int_0^T|Z_s|^2ds<\infty$.
We shall also denote by ${\mathbb S}^{2,n}$
the space of all continuous processes ${\boldsymbol S}=(S_{t})_{0 \leq t \leq T}$ 
in $\HH^{0,n}$ 
such that $\EE[ \sup_{0 \leq t \leq T} \vert S_{t} \vert^2] < + \infty$.
Here and in what follows, regularity properties, such as continuity or Lipschitz character, of functions of measures are always understood in the sense of the $2$-Wasserstein distance of these respective spaces of probability measures.

Throughout this section, we assume:
\begin{description}
\item[(I)]The functions $b$, $\sigma$ and 
$f$ are differentiable with respect to $(x,\alpha)$, for $\xi\in\cP_{2}(\RR^d\times\RR^k)$ fixed, and the functions $(x,\alpha,\xi) \mapsto \left(\,
\partial_{x} (b,\sigma,f) (x,\alpha,\xi),
\partial_{\alpha} (b,\sigma,f) (x,\alpha,\xi) \,\right)$ are continuous.
Moreover, the functions $b$, $\sigma$ and $f$ are L-differentiable with respect to the variable $\xi$, the mapping  $$\RR^d\times A \times  L^2(\Omega,\cF,\PP;\RR^d\times \RR^k) \ni (x,\alpha,(X,\beta)) \mapsto \partial_{\mu} (b,\sigma,f)(x,\alpha,\cL(X,\beta))(X,\beta)$$  being continuous. Similarly, the function $g$ is differentiable with respect to $x$, the mapping 
$(x,\mu) \mapsto \partial_{x} g(x,\mu)$ being continuous. The function $g$ is also L-differentiable with respect to the variable $\mu$, and the following map is continuous 
$$
\RR^d \times L^2(\Omega,\cF,\PP;\RR^d) \ni (x,X) \mapsto \partial_{\mu} g(x,\cL(X))(X) \in L^2(\Omega,\cF,\PP;\RR^d).
$$

\item[(II)]The  derivatives $\partial_{x} (b,\sigma)$ and $\partial_{\alpha} (b,\sigma)$ are uniformly bounded, and the mapping $(x',\alpha')\mapsto
\partial_{\mu}(b,\sigma)(x,\alpha,\xi)(x',\alpha')$ (resp. $(x',\alpha') \mapsto
\partial_{\nu}(b,\sigma)(x,\alpha,\xi)(x',\alpha')$) has an $L^2(\RR^d,\mu;\RR^d\times\RR^k)$-norm (resp. $L^2(\RR^k,\nu;\RR^d\times\RR^k)$-norm) which is 
uniformly bounded in $(x,\alpha,\xi)$. There exists a constant $L$ such that, for any $R \geq 0$ and any 
$(x,\alpha,\xi)$ such that $\vert x \vert,\vert \alpha \vert,\|\xi\|_{L^2} \leq R$, it holds that $$\vert \partial_{x} f(x,\alpha,\xi) \vert \vee\vert \partial_{x} g(x,\mu) \vert \vee\vert \partial_{\alpha} f(x,\alpha,\xi) \vert\leq L(1+R),$$ and the norms in $L^2(\RR^d\times\RR^k,\xi;\RR^d\times\RR^k)$ and $L^2(\RR^d,\xi;\RR^d\times\RR^k)$ of 
$(x',\alpha') \mapsto \partial_{\mu}f(x,\alpha,\xi)(x',\alpha')$, $(x',\alpha') \mapsto \partial_{\nu}f(x,\alpha,\xi)(x',\alpha')$ and $x' \mapsto
\partial_{\mu}g(x,\mu)(x')$ are bounded by $L(1+R)$. 
\end{description}

\vskip 4pt\noindent
Under these assumptions, for any admissible control $\balpha\in{\mathbb A}$, we denote by $ \bX=  \bX^{\alpha}$ the corresponding controlled state process satisfying \eqref{fo:state}. We call adjoint processes of $\bX$ (or of the admissible control $\balpha$), {the} couple  $(\bY,\bZ)$ of  stochastic processes $\bY=(Y_t)_{0\le t\le T}$ and $\bZ=(Z_{t})_{0 \leq t \leq T}$ in ${\mathbb S}^{2,d} \times \HH^{2,d\times m}$ satisfying: 
\begin{eqnarray}
\label{fo:adjoint}
\begin{cases}
&dY_t=-
\Bigl[ \partial_xH\bigl(\theta_t,Y_t,Z_t\bigr)
+
\tilde{\mathbb E}\bigl[\partial_\mu  H\bigl(\tilde \theta_t,\tilde Y_t,\tilde Z_t)(X_t,\alpha_t)\bigr]
\Bigr] dt
+ Z_t dW_t, \quad t \in [0,T],
\vspace{5pt}
\\
& Y_T= \partial_xg\bigl(X_T,\cL(X_T)\bigr)+\tilde{\mathbb E}\bigl[\partial_\mu g\bigl(\tilde X_T,\cL(X_T)\bigr)(X_T) \bigr],
\end{cases}
\end{eqnarray}
where we set $\theta_t=(X_t,\alpha_t,\cL(X_t,\alpha_t))$, and the tilde notation refers to an independent copy.
Equation \eqref{fo:adjoint} is referred to as the adjoint equation. Given $\balpha$ and $\bX$, it is a backward stochastic differential equation which is well posed under the current assumptions. 

\subsection{\textbf{A Necessary Condition}}
The main result of this subsection is based on the following expression of the G\^ateaux derivative of the cost function $J(\balpha)$.

\begin{lem}\label{fo:Hgateaux}
{Let $\balpha\in{\mathbb A}$, $\bX$ be the corresponding controlled state process, and $(\bY,\bZ)$ its adjoint processes. 
For $\bbeta\in{\mathbb A}$, the G\^ateaux derivative of $J$ at $\balpha$ in the direction $\bbeta-\balpha$ is
\begin{equation*}
\frac{d}{d\epsilon}J(\balpha+\epsilon(\bbeta-\balpha))\big|_{\epsilon=0} ={\mathbb E} \int_0^T \Bigl( \partial_\alpha H(\theta_t,Y_t,Z_t)
+\tilde {\mathbb E}[\partial_\nu H(\tilde\theta_t,\tilde Y_t,\tilde Z_t)(X_t,\alpha_t)]\Bigr) \cdot (\beta_t-\alpha_t) \; dt,
\end{equation*}
}
where $(\tilde{\bX},\tilde{\bY},\tilde{\bZ},\tilde\balpha,\tilde\bbeta)$ is an independent copy of $(\bX,\bY,\bZ,\balpha,\bbeta)$ on the space $L^2(\tilde{\Omega},\tilde {\mathcal F},\tilde \PP)$.
\end{lem}

\begin{proof}
We follow the lines of the proof of the stochastic maximum principle for the control of McKean-Vlasov equations given in Section 6.3 of \cite{CarmonaDelarue_book_vol_I}. To lighten the notation, we only give the details in the case $\sigma\equiv I_d$ and $g\equiv 0$.
Given admissible $\balpha$ and $\bbeta$, for each $\epsilon>0$ we define the admissible control  $\balpha^\epsilon=(\alpha^\epsilon_t)_{0\le t\le T}$ by {$\alpha^\epsilon_t=\alpha_t+\epsilon(\beta_t-\alpha_t)$}, and we denote by $\bX^\epsilon=(X^\epsilon_t)_{0\le t\le T}$ the solution of the state equation \eqref{fo:state} for the control $\balpha^\epsilon$ in lieu of $\balpha$. We then define the process $\bV=(V_t)_{0\le t\le T}$ as the G\^ateaux derivative of the state in the direction {$\eeta:=\bbeta-\balpha$}.
In other words, we define $V_t:=\lim_{\epsilon\searrow 0}\epsilon^{-1}(X^\epsilon_t-X_t)$. It is easy to check that this process satisfies:
$$
dV_t=\Bigl( \partial_x b(\theta_t)V_t +\partial_\alpha b(\theta_t)\eta_t + \tilde{\mathbb E}[\partial_\mu b(\theta_t)(\tilde X_t,\tilde\alpha_t)\tilde V_t] +
\tilde{\mathbb E}[\partial_\nu b(\theta_t)(\tilde X_t,\tilde\alpha_t)\tilde \eta_t]\Bigr)\,dt,
$$
with zero initial condition. In the same way we obtain:
\begin{eqnarray*}
&&\lim_{\epsilon\searrow 0}\frac{1}{\epsilon}[J(\balpha^\epsilon)-J(\balpha)]\\
&&\hskip 12pt
={\mathbb E}\int_0^T\Bigl( \partial_x f(\theta_t)V_t +\partial_\alpha f(\theta_t)\eta_t + \tilde{\mathbb E}[\partial_\mu f(\theta_t)(\tilde X_t,\tilde\alpha_t)\tilde V_t] +
\tilde{\mathbb E}[\partial_\nu f(\theta_t)(\tilde X_t,\tilde\alpha_t)\tilde \eta_t]\Bigr)\,dt\\
&&\hskip 12pt
={\mathbb E}\int_0^T\Bigl( \partial_x f(\theta_t)V_t +\partial_\alpha H(\theta_t,Y_t)\eta_t -\partial_\alpha b(\theta_t)Y_t\eta_t\\
&&\hskip 75pt
+ \tilde{\mathbb E}[\partial_\mu f(\theta_t)(\tilde X_t,\tilde\alpha_t)\tilde V_t] +
\tilde{\mathbb E}[\partial_\nu f(\theta_t)(\tilde X_t,\tilde\alpha_t)\tilde \eta_t]\Bigr)\,dt,
\end{eqnarray*}
where we used the definition of the Hamiltonian. In the present situation, $H$ is independent of $Z$ because we assume $\sigma$ is not controlled. To be more specific, here:
$$
H(x,\alpha,\xi,y)=b(x,\alpha,\xi)\cdot y  + f(x,\alpha,\xi).
$$
Note that, at this stage, the process $\bY=(Y_t)_{0\le t\le T}$ could be any process. In what follows, we choose this process to be the adjoint process of the control $\balpha$. In this case, if we use the form \eqref{fo:adjoint} of the adjoint equation, we get:
\begin{eqnarray*}
0&=&{\mathbb E}[Y_TV_T]={\mathbb E}\int_0^TY_t\,dV_t + {\mathbb E}\int_0^T V_t\,dY_t\\
&=&{\mathbb E}\int_0^T\Bigl( Y_t\partial_x b(\theta_t)V_t +Y_t\partial_\alpha b(\theta_t)\eta_t + Y_t\tilde{\mathbb E}[\partial_\mu b(\theta_t)(\tilde X_t,\tilde\alpha_t)\tilde V_t] +
Y_t\tilde{\mathbb E}[\partial_\nu b(\theta_t)(\tilde X_t,\tilde\alpha_t)\tilde \eta_t]\Bigr)\,dt\\
&&\hskip 25pt
-V_t\partial_x b(\theta_t)Y_t -V_t\partial_x f(\theta_t) 
-V_t\tilde{\mathbb E}[\partial_\mu b(\tilde\theta_t)(X_t,\alpha_t)\tilde Y_t] 
-V_t\tilde{\mathbb E}[\partial_\mu f(\tilde\theta_t)(X_t,\alpha_t)]\Bigr)\,dt\\
&=&{\mathbb E}\int_0^T\Bigl( Y_t\partial_\alpha b(\theta_t)\eta_t + Y_t\tilde{\mathbb E}[\partial_\mu b(\theta_t)(\tilde X_t,\tilde\alpha_t)\tilde V_t] +
Y_t\tilde{\mathbb E}[\partial_\nu b(\theta_t)(\tilde X_t,\tilde\alpha_t)\tilde \eta_t]\Bigr)\,dt\\
&&\hskip 25pt
-V_t\partial_x f(\theta_t) 
-V_t\tilde{\mathbb E}[\partial_\mu b(\tilde\theta_t)(X_t,\alpha_t)\tilde Y_t] 
-V_t\tilde{\mathbb E}[\partial_\mu f(\tilde\theta_t)(X_t,\alpha_t)]\Bigr)\,dt.
\end{eqnarray*}
From this we derive
\begin{eqnarray*}
{\mathbb E}\int_0^T\Bigl( \partial_x f(\theta_t)V_t  -\partial_\alpha b(\theta_t)Y_t\eta_t\Bigr)dt
&=&
{\mathbb E}\int_0^T\Bigl(  Y_t\tilde{\mathbb E}[\partial_\mu b(\theta_t)(\tilde X_t,\tilde\alpha_t)\tilde V_t] +
Y_t\tilde{\mathbb E}[\partial_\nu b(\theta_t)(\tilde X_t,\tilde\alpha_t)\tilde \eta_t]\Bigr)\,dt\\
&&\hskip 25pt
-V_t\tilde{\mathbb E}[\partial_\mu b(\tilde\theta_t)(X_t,\alpha_t)\tilde Y_t] 
-V_t\tilde{\mathbb E}[\partial_\mu f(\tilde\theta_t)(X_t,\alpha_t)]\Bigr)\,dt.
\end{eqnarray*}
Substituting this in the last expression we found for the G\^ateaux derivative of $J$, we get:
\begin{eqnarray*}
\lim_{\epsilon\searrow 0}\frac{1}{\epsilon}[J(\balpha^\epsilon)-J(\balpha)] 
&=&{\mathbb E}\int_0^T\Bigl(\partial_\alpha H(\theta_t,Y_t)\eta_t + Y_t\tilde{\mathbb E}[\partial_\mu b(\theta_t)(\tilde X_t,\tilde\alpha_t)\tilde V_t] +
Y_t\tilde{\mathbb E}[\partial_\nu b(\theta_t)(\tilde X_t,\tilde\alpha_t)\tilde \eta_t]\\
&&\,\,\,\,\,\,\,\,\,\,
-V_t\tilde{\mathbb E}[\partial_\mu b(\tilde\theta_t)(X_t,\alpha_t)\tilde Y_t] 
-V_t\tilde{\mathbb E}[\partial_\mu f(\tilde\theta_t)(X_t,\alpha_t)] \\
&&\,\,\,\,\,\,\,\,\,\, + \tilde{\mathbb E}[\partial_\mu f(\theta_t)(\tilde X_t,\tilde\alpha_t)\tilde V_t] +
\tilde{\mathbb E}[\partial_\nu f(\theta_t)(\tilde X_t,\tilde\alpha_t)\tilde \eta_t]\Bigr)\,dt\\
&=&{\mathbb E}\int_0^T\Bigl(\partial_\alpha H(\theta_t,Y_t)\eta_t +
\tilde{\mathbb E}[\partial_\nu H(\theta_t,Y_t)(\tilde X_t,\tilde\alpha_t)\tilde \eta_t]\Bigr)\,dt\\
&&\,\,\,\,\,\,\,\,\,\,+\int_0^T\Bigl( -{\mathbb E}\tilde{\mathbb E}[V_t\partial_\mu b(\tilde\theta_t)(X_t,\alpha_t)\tilde Y_t] 
+ {\mathbb E}\tilde{\mathbb E}[Y_t\partial_\mu b(\theta_t)(\tilde X_t,\tilde\alpha_t)\tilde V_t]\Bigr)\,dt \\
&&\,\,\,\,\,\,\,\,\,\,
+ \int_0^T\Bigl({\mathbb E}\tilde{\mathbb E}[\partial_\mu f(\theta_t)(\tilde X_t,\tilde\alpha_t)\tilde V_t] 
-{\mathbb E}\tilde{\mathbb E}[V_t\partial_\mu f(\tilde\theta_t)(X_t,\alpha_t)]\Bigr)\,dt,
\end{eqnarray*}
which gives the desired result because the integrals on the last two lines are both equal to $0$, given that the tilde notations signify independent copies.
\end{proof}

\vskip 6pt
We are now ready to prove the necessary part of the Pontryagin stochastic maximum principle. {In the present framework of extended Mean Field control we obtain \eqref{fo:Pontryagin_sufficient} below.  It is not possible to improve this condition into a pointwise minimization condition as in more classical versions of the problem when there is no non-linear dependence on the law of the control{, see (6.58) in \cite{CarmonaDelarue_book_vol_I}}. We give an example of this phenomenon in  Remark \ref{rem example}. 
}

\begin{thm}
\label{thm:converse:pontryagin}
Under Assumptions (I)-(II), if the admissible control 
$\balpha=(\alpha_{t})_{0 \leq t \leq T} \in{\mathbb A}$ is optimal,  $\bX=(X_{t})_{0 \leq t \leq T}$ is the associated controlled state given by \eqref{fo:state}, and  
$(\bY,\bZ)=( Y_{t},  Z_{t})_{0 \leq t \leq T}$ are the associated adjoint processes satisfying \eqref{fo:adjoint}, then we have:
\begin{equation}
\label{fo:Pontryagin_sufficient}
\left (\partial_{\alpha} H(\theta_t,Y_t,Z_t)+
\tilde\EE \bigl[ \partial_{\nu} H(\tilde\theta_t,\tilde Y_t,\tilde Z_t) (X_t,\alpha_t)\bigr]\right ) \cdot (\alpha_t-a)\leq \,0\quad  \forall a\in A, \,\textrm{\rm Leb}_{1}
 \otimes \PP \;\text{a.e.},
\end{equation}
where $(\tilde{\bX},\tilde{\bY},\tilde{\bZ},\tilde\balpha)$ is an independent copy of $(\bX,\bY,\bZ,\balpha)$ on $L^2(\tilde{\Omega},\tilde {\mathcal F},\tilde \PP)$.
\end{thm}
\begin{proof}
Given any admissible control $\bbeta$, we use as before the perturbation $\alpha^\epsilon_t=\alpha_t+\epsilon (\beta_t-\alpha_t)$.
Since $\balpha$ is optimal, we have the inequality
$$\textstyle
\frac{d}{d\epsilon}J(\balpha+\epsilon(\bbeta-\balpha))\big|_{\epsilon=0} \ge 0.
$$
Using the result of the previous lemma, we get:
\begin{equation*}\label{eq:first}\textstyle
{\mathbb E} \int_0^T \Bigl( \partial_\alpha H(\theta_t,Y_t,Z_t)
+\tilde {\mathbb E}[\partial_\nu H(\tilde\theta_t,\tilde Y_t,\tilde Z_t)(X_t,\alpha_t)]\Bigr)\cdot ( \beta_t -\alpha_t) \; dt \ge 0.
\end{equation*}
We now use the same argument as in the classical case (see e.g. \cite[Theorem 6.14]{CarmonaDelarue_book_vol_I}). For every $t$ and $\mathcal{F}_t$-measurable $\beta$ with values in $A$, we can take {$\beta_t$
equal to $\alpha_t$}
except for the interval $[t,t+\varepsilon]$ where it equals $\beta$, obtaining
\begin{equation}\label{eq:not first}\textstyle
{\mathbb E} \Bigl( \partial_\alpha H(\theta_t,Y_t,Z_t)
+\tilde {\mathbb E}[\partial_\nu H(\tilde\theta_t,\tilde Y_t,\tilde Z_t)(X_t,\alpha_t)]\cdot ( \beta -\alpha_t)  \Bigr)\ge 0.
\end{equation}
Further, for any $a\in A$ we can take $\beta$ to be equal to $a$ on an arbitrary set in $\mathcal{F}_t$, and to coincide with $\alpha_t$ otherwise, establishing equation \eqref{fo:Pontryagin_sufficient}.
\end{proof}

\begin{rem}
If $A$ is open, or if the admissible optimal control $\balpha$ takes values in the interior of $A$, then we may replace \eqref{fo:Pontryagin_sufficient} with the following (see e.g.\ \cite[Proposition 6.15]{CarmonaDelarue_book_vol_I}):
\begin{equation}\label{fo:Pontryagin_sufficient0}
\partial_{\alpha} H(\theta_t,Y_t,Z_t)+ \tilde\EE[\partial_{\nu} H(\tilde\theta_t,\tilde Y_t,\tilde Z_t)(X_t,\alpha_t)]=0 { \qquad \,\,\textrm{\rm Leb}_{1}
 \otimes \PP \;\text{a.e.}}.
 \end{equation}
\end{rem}
\begin{rem}
A sharpening of \eqref{fo:Pontryagin_sufficient} can be obtained  under the convexity condition:
\begin{equation*}
\label{fo:H_convexity_a}
\begin{split}
H(x,a',\xi',Y_t, Z_t) &\geq H(x,a,\xi,Y_t,Z_t) 
+
\partial_{\alpha} H(x,a,\xi,Y_t,Z_t) \cdot (a'-a)
 \\  
&\hspace{45pt} +
\tilde\EE \bigl[ \partial_{\nu} H(x,a,\xi,Y_t,Z_t) (\tilde X_t,\tilde\alpha_t)
\cdot (\tilde \alpha'_t- \tilde \alpha_t) \bigr],
\end{split}
\end{equation*}
$\textrm{\rm Leb}_{1}
 \otimes \PP \;\text{almost everywhere}$,
for all $x\in\RR^d$, $a,a'\in A$, and $\xi,\xi'\in{\mathcal P}_{2}(\RR^d\times A)$ with $\xi=\cL(\tilde X_t,\tilde\alpha_t)$ and $\xi'=\cL(\tilde X_t,\tilde\alpha'_t)$. Indeed, under this condition and  from \eqref{eq:not first}, we get
$\EE \left [ H(X_t,\beta,\mathcal{L}(X_t,\beta),Y_t,Z_t) \right ] \geq \EE \left [ H(X_t,\alpha_t,\mathcal{L}(X_t,\alpha_t),Y_t,Z_t) \right ]  $ in the above proof, so
\begin{align*}
\label{eq strange}
\alpha_t = \text{argmin}\,\,\left\{\,\, \EE \left [ H(X_t,\beta,\mathcal{L}(X_t,\beta),Y_t,Z_t) \right ]\,:\,\, \beta\in L^2(\Omega,\mathcal{F}_t,\PP{; A}) \,\, \right\}.
\end{align*}
\end{rem}

\subsection{\textbf{A Sufficient Condition}}
\label{sub:mkv_sufficient}
Guided by the necessary condition proven above, we derive a sufficient condition for opti\-mality in the same spirit, though under stronger convexity assumptions. These read as
\begin{equation}
\label{fo:g_convexity}
g(x,\mu)-g(x',\mu')\le  \partial_xg(x,\mu)\cdot (x-x')+\tilde{\mathbb E} \bigl[\partial_\mu g(x,\mu)(\tilde X)\cdot (\tilde X-\tilde {X'})\bigr],
\end{equation}
for all $x,x'\in\RR^d$, and $\mu,\mu'\in{\mathcal P}_{2}(\RR^d)$, with $\mu=\cL(\tilde X)$ and $\mu'=\cL(\tilde X')$,
see \cite[Ch.6]{CarmonaDelarue_book_vol_I}, and
\begin{equation}
\label{fo:H_convexity}
\begin{split}
H(x',a',\xi',y,z) &\geq H(x,a,\xi,y,z) 
 +  \partial_{x} H(x,a,\xi,y,z) \cdot (x'-x)
+
\partial_{\alpha} H(x,a,\xi,y,z) \cdot (a'-a)
 \\ 
& +
\tilde\EE \bigl[ \partial_{\mu} H(x,a,\xi,y,z) (\tilde X,\tilde\alpha)
\cdot (\tilde X'- \tilde X) + \partial_{\nu} H(x,a,\xi,y,z) (\tilde X,\tilde\alpha)
\cdot (\tilde \alpha'- \tilde \alpha) \bigr],
\end{split}
\end{equation}
for all $x,x'\in\RR^d$, $a,a'\in A$, $y\in\RR^d$, $z\in\RR^{d\times m}$, and $\xi,\xi'\in{\mathcal P}_{2}(\RR^d\times A)$ with $\xi =  \mathcal{L} (\tilde X,\tilde \alpha)$ and $\xi' =  \mathcal{L} (\tilde X',\tilde \alpha')$.

\begin{thm}
\label{th:pontryagin}
Under Assumptions (I)-(II), let $\balpha=(\alpha_{t})_{0 \leq t \leq T}\in{\mathbb A}$ be an admissible control, 
$\bX= (X_{t})_{0 \leq t \leq T}$ the corresponding controlled state process, and $(\bY,\bZ)=(Y_{t}, Z_{t})_{0 \leq t \leq T}$ the corresponding adjoint processes.
Let us  assume that:
\begin{enumerate}
\item $g$ is convex in the sense of \eqref{fo:g_convexity};
\item $\RR^d \times A\times {\mathcal P}_{2}(\RR^d\times A)  \ni (x,\alpha,\xi)\mapsto H(x,\alpha,\xi,Y_t,Z_t)$ is convex in the sense of \eqref{fo:H_convexity},
$\textrm{\rm Leb}_{1}
 \otimes \PP \;\text{almost everywhere}$.
\end{enumerate} 
Then, if
\eqref{fo:Pontryagin_sufficient} holds,
$\balpha$ is an optimal control, i.e.
$J(\balpha)=\inf_{ \balpha'\in{\mathbb A}}J( \balpha')
$. 
\end{thm}

As before, we use the notation $\theta_t=(X_t,\alpha_t,\cL(X_t,\alpha_t))$ throughout the proof.
 
\begin{proof}
We follow the steps of the classical proofs; see for example \cite[Theorem 6.16]{CarmonaDelarue_book_vol_I} for the case of the control of standard McKean-Vlasov SDEs. Let $\balpha'\in{\mathbb A}$ be any admissible control, and $ \bX'= \bX^{\balpha'}$ the corresponding controlled state.
By definition of the objective function and the Hamiltonian of the control problem, we have:
\begin{align}
\label{fo:objective_diff}
J(\balpha)-J(\balpha')
&= \textstyle {\mathbb E} \bigl[g(X_T,\cL(X_T))-g(X'_T,\cL(X'_T)) \bigr]+ {\mathbb E}\int_0^T
\bigl[f(\theta_t)-f(\theta_t') \bigr]dt\nonumber
\\
&= \textstyle{\mathbb E} \bigl[g(X_T,\cL(X_T))-g(X'_T,\cL(X'_T)) \bigr] +{\mathbb E}\int_0^T \bigl[H(\theta_t,Y_t,Z_t)-H(\theta_t',Y_t,Z_t) \bigr]dt\nonumber
\\
&\textstyle \phantom{????????}- {\mathbb E}\int_0^T \bigl\{ \bigl[b(\theta_t)-b(\theta_t') \bigr]\cdot Y_t + 
\bigl[\sigma(\theta_t)-\sigma(\theta'_t)]\cdot Z_t \bigr\} dt,
\end{align}
with $\theta'_{t} = (X'_{t},\alpha'_{t},\cL(X'_t,\alpha'_t))$. Being $g$ convex, we have: 
\begin{equation}
\label{fo:g}
\begin{split}
&{\mathbb E}\bigl[g\bigl(X_T,\cL(X_T)\bigr)-g\bigl(X'_T,\cL(X'_T)\bigr)\bigr]
\\
&\hspace{25pt}\le {\mathbb E} \bigl[\partial_xg(X_T,\cL(X_T)) \cdot (X_T-X'_T)  + \tilde{\mathbb E} \bigl[\partial_\mu g(X_T,\cL(X_t))(\tilde X_T)  
\cdot (\tilde X_T-\tilde {X'_T}) \bigr]
\bigr]
\\
&\hspace{25pt}= {\mathbb E} \bigl[ \bigl( \partial_xg(X_T,\cL(X_T))+ \tilde{\mathbb E}[\partial_\mu g(\tilde X_T, \cL(X_t))(X_T) ] \bigr) \cdot (X_T- X'_T) \bigr]
\\
&\hspace{25pt}= {\mathbb E} \bigl[(X_T-X'_T)\cdot Y_T \bigr],
\end{split}
\end{equation}
where we used Fubini and the fact that the `tilde random variables' are independent copies of the `non-tilde' ones. By the adjoint equation and taking the expectation, we get:
\begin{align}
\label{fo:x_diff}
&\textstyle{\mathbb E} \left[ (X_T-X'_T)\cdot Y_T \bigr] 
={\mathbb E} \biggl[ \int_0^T(X_t-X'_t) \cdot dY_t+\int_0^TY_t \cdot d[X_t- X'_t]
+\int_0^T[\sigma(\theta_t)-\sigma(\theta_t')]\cdot  Z_tdt \right]\nonumber
\\
&=\textstyle - {\mathbb E} \int_0^T
\bigl[ 
\partial_xH(\theta_{t},Y_t,Z_t) \cdot (X_t-X'_t) + 
\tilde{\mathbb E} \bigl[ \partial_\mu H(\tilde \theta_t,\tilde Y_t,\tilde Z_t)(X_t,\alpha_t) \bigr]  \cdot (X_t-X'_t)  \bigr] dt\nonumber
\\
&\textstyle \phantom{??????}+{\mathbb E} \int_0^T 
\bigl[
[b(\theta_t)-b(\theta'_t)] \cdot Y_{t}   + [\sigma(\theta_t)-\sigma(\theta'_t)]\cdot  Z_t \bigr] dt,
\end{align}
where we used integration by parts and the fact that $\bY=(Y_t)_{0 \leq t \leq T}$ solves the adjoint equation \eqref{fo:adjoint}. 
Again by Fubini's theorem, we get:
\begin{equation*}
\begin{split} \textstyle
{\mathbb E} \int_0^T  \tilde{\mathbb E} \bigl[ \partial_\mu H(\tilde \theta_t,\tilde Y_t,\tilde Z_t)(X_t,\alpha_t) \bigr] \cdot (X_t-X'_t)  dt
&\textstyle ={\mathbb E} \int_0^T \tilde{\mathbb E} \bigl[ \partial_\mu H(\theta_t,Y_t,Z_t)(\tilde X_t,\tilde\alpha_t)\cdot (\tilde X_t - \tilde X'_t) \bigr] dt.
\end{split}
\end{equation*}
Together with \eqref{fo:objective_diff}, \eqref{fo:g}, and \eqref{fo:x_diff}, this gives:
\begin{align*}
&J(\balpha)-J(\balpha')
\le \textstyle{\mathbb E}\int_0^T[H(\theta_t,Y_t,Z_t)-H(\theta'_t,Y_t,Z_t)]dt
\\
& \textstyle\phantom{????????????}-{\mathbb E}\int_0^T \Bigl[ \partial_xH(\theta_t,Y_t,Z_t) \cdot (X_t-X'_t) + 
\tilde{\mathbb E} \bigl[ \partial_\mu H( \theta_t,Y_t,Z_t)(\tilde X_t,\tilde \alpha_t) \cdot (\tilde X_{t} - \tilde X'_{t} ) \bigr] \Bigr] dt
\\
&\textstyle\phantom{????????????}\le {\mathbb E}\int_0^T \Bigl[ \partial_\alpha H(\theta_t,Y_t,Z_t) \cdot (\alpha_t-\alpha'_t) + 
\tilde{\mathbb E} \bigl[ \partial_\nu H(\theta_t,Y_t, Z_t)(\tilde X_t,\tilde\alpha_t) \cdot (\tilde \alpha_{t} - \tilde \alpha'_{t} ) \bigr] \Bigr] dt
\\
&\textstyle\phantom{????????????}= {\mathbb E}\int_0^T \Bigl( \partial_\alpha H(\theta_t,Y_t,Z_t) + \tilde\EE\bigl[ \partial_\nu H(\tilde\theta_t,\tilde Y_t, \tilde Z_t)(X_t,\alpha_t)\bigr] \Bigr)\cdot (\alpha_t-\alpha'_t)dt\, 
\leq 0,
\end{align*}
because of the convexity of $H$, Fubini's theorem, and \eqref{fo:Pontryagin_sufficient}, proving that $\balpha$ is optimal.
\end{proof}

\section{\textbf{Examples}}
\label{sect:examples}
In this section, we consider models for which the solution strategy suggested by the stochastic maximum principle proved in the previous section can be pushed further. In fact, in Sections~\ref{sub:mkv_liq} and \ref{sub:mkv_LQ}, we actually push the analysis all the way to explicit solutions. 

\subsection{\textbf{The Case of Scalar Interactions}}
\label{sub:mkv_scalar}
For the sake of definiteness, we state explicitly what the above forms of the Pontryagin stochastic maximum principle  become in the case of scalar interactions. This particular case is often studied first because it can be dealt with using standard calculus since it does not really need the full generality of the differential calculus on Wasserstein spaces of measures. In the next subsection we will study a special example of scalar interactions, in an optimal liquidation problem, where we provide explicit solutions; see also \cite{G} for another application of scalar interactions.

In this subsection, the drift function $b$ (resp.\ the running cost function $f$) is of the form
$$
\textstyle{b(x,\alpha,\xi)=b_0\left(x,\alpha,\int \varphi d\xi\right),\qquad \left(\text{resp. }f(x,\alpha,\xi)=f_0\left(x,\alpha,\int\psi d\xi\right)\;\right)}
$$
for some functions $\RR^d\times A\times \RR\ni (x,\alpha,\zeta)\mapsto b_0(x,\alpha,\zeta)\in\RR$ and $\RR^d\times A\ni (x,\alpha)\mapsto \varphi(x,\alpha)\in\RR$ (resp. $\RR^d\times A\times \RR\ni (x,\alpha,\zeta)\mapsto f_0(x,\alpha,\zeta)\in\RR$ and $\RR^d\times A\ni (x,\alpha)\mapsto \psi(x,\alpha)\in\RR$). 
Similarly, we assume that 
$g(x,\mu)=g_0\left(x,\int \phi d\mu\right)$ for some functions $\RR^d\times\RR\ni(x,\zeta)\mapsto g_0(x,\zeta)\in\RR$ and  $\RR^d\ni x\mapsto \phi(x)\in\RR$.
In order to simplify the notation, we shall assume that the volatility is independent of the control, and actually we take $\sigma\equiv I_d$. 

\vskip 2pt
Under these circumstances, the adjoint equation becomes:
\begin{equation*}
\label{fo:scalar_adjoint}
\begin{split}
dY_t&=-\Bigl(\partial_xb_0(X_t,\alpha_t,\EE[\varphi(X_t,\alpha_t)])Y_t +\partial_xf_0(X_t,\alpha_t,\EE[\psi(X_t,\alpha_t)]) \\
&\hspace{55pt}+\tilde\EE[\tilde Y_t\cdot \partial_\zeta b_0(\tilde X_t,\tilde \alpha_t,\EE[\varphi(X_t,\alpha_t)])]\partial_x\varphi(X_t,\alpha_t)\\
&\hspace{75pt}+\tilde\EE[\partial_\zeta f_0(\tilde X_t,\tilde \alpha_t,\EE[\psi(X_t,\alpha_t)])]\partial_x\psi(X_t,\alpha_t)
\Bigr)dt + Z_t dW_t\, ,
\end{split}
\end{equation*}
with terminal condition $Y_T=\partial_xg_0\left (X_T,\EE[\phi(X_T)]\right ) + \tilde\EE[\partial_\zeta g_0(\tilde X_T,\EE[\phi(X_T)])]\partial_x\phi(X_T)$. Accordingly, the necessary condition \eqref{fo:Pontryagin_sufficient0} for optimality will be satisfied when
\begin{align}
\label{fo:scalar_sufficient}
0&=\partial_\alpha b_0(X_t,\alpha_t,\EE[\varphi(X_t,\alpha_t)])\cdot Y_t +\partial_\alpha f_0(X_t,\alpha_t,\EE[\psi(X_t,\alpha_t)]) \\
&+\tilde\EE[\tilde Y_t\cdot \partial_\zeta b_0(\tilde X_t,\tilde \alpha_t,\EE[\varphi(X_t,\alpha_t)])]\partial_\alpha \varphi(X_t,\alpha_t+\tilde\EE[\partial_\zeta f_0(\tilde X_t,\tilde \alpha_t,\EE[\psi(X_t,\alpha_t)])]\partial_\alpha \psi(X_t,\alpha_t).\nonumber
\end{align}

\subsection{Optimal liquidation with market impact}\label{sub:mkv_liq}
We explicitly solve an example which lies outside the classical LQ framework, in the sense that convexity fails. This is inspired by an optimal liquidation problem with price impact, but due to its simplicity it is more of mathematical interest than financial one.

Consider a market where a group of agents have a position on a certain asset which they want to liquidate by a fixed time $T>0$. Trades of all market participants reflect on a permanent and a temporary market impact. The optimal trades will be a result of the trade-off between trading slowly to reduce the market impact (or execution/liquidity cost), and trading fast to reduce the risk of future uncertainty in prices; see e.g.\ \cite{almgren2001optimal,CJ,CL,CL15,BP}.
We formulate the asymptotic problem (infinite number of agents) in the case of cooperative equilibria.
The price process is modeled by
\[
dS_t=\lambda\EE[\alpha_t]dt+\sigma dW_t,\quad S_0=s_0,
\]
where $\lambda\EE[\alpha_t]$ is the permanent market impact to which all agents contribute. Naturally $\lambda\geq 0$. The inventory process evolves as
\begin{equation}\label{eq_inv}
dQ_t=\alpha_tdt, 
\end{equation}
with $Q_0$ {(possibly random)} being the initial inventory to deplete by time $T$.
The wealth process is given by
\[
dX_t=-\alpha_t(S_t+k\alpha_t)dt, \quad X_0=0,
\]
where $k\alpha_t$ is the temporary market impact each agent is subject to (we can think of it as fees/liquidity cost).
The cost to be minimized is
\[\textstyle
\EE\left[\phi\int_0^T Q_t^2dt-Q_T(S_T-AQ_T)-X_T\right],
\]
where $\phi$ is a risk aversion parameter, $Q_T(S_T-AQ_T)$ is the liquidation value of the remaining quantity at terminal time (with a liquidation/execution penalization), and $X_T$ is the terminal profit due to trading in $[0,T]$.
Using the dynamics of $X$, this can be rewritten as
\[\textstyle
\EE\left[\int_0^T (\phi Q_t^2+\alpha_tS_t+k\alpha_t^2)dt-Q_T(S_T-AQ_T)\right].
\]
So here we have a 2-dimensional state process $(S,Q)$, a 1-dimensional Wiener process $W$, and the control process is the trading speed $\alpha$. The Hamiltonian of the system is
\[
H(x_1,x_2,a,\xi,y_1,y_2)=\lambda\bar\xi_2y_1+ay_2+\phi x_2^2+ax_1+ka^2,
\]
where $\bar\xi_2=\int v\xi(du,dv)$, and the first order condition \eqref{fo:scalar_sufficient} reads as
\begin{equation}\label{FOC_lo}
Y_t^2+S_t+2k\alpha_t+\lambda\EE[Y_t^1]=0,
\end{equation}
with the adjoint equations being given by
\begin{eqnarray}\label{eq_adj1}
dY_t^1&=&-\alpha_tdt+Z_t^{1}dW_t,\qquad Y_T^1=-Q_T\\
dY_t^2&=&-2\phi Q_tdt+Z_t^{2}dW_t,\hspace*{0.3cm} Y_T^2=-S_T+2AQ_T.\label{eq_adj2}
\end{eqnarray}

\begin{rem}
The convexity assumption \eqref{fo:g_convexity} on $g$ is not satisfied. However, when $A\ge \lambda/2$ (satisfied for typical values of the parameters; see \cite{CJ,CL}), the cost function can be rearranged so that convexity is recovered, thus guaranteeing that the first order condition \eqref{FOC_lo} is not only necessary but also sufficient for the optimality of $\alpha$. For instance, by integration by parts for the term $\EE[Q_TS_T]$, one can obtain that the cost of the problem consists of a convex part (in the usual sense) plus the term $(A-\lambda/2)\EE[(Q_T)^2]+(\lambda/2) \text{Var}(Q_T)$, the second part of which is convex in the sense of \eqref{fo:g_convexity}, as can be easily seen.
\end{rem}

An inspection of \eqref{eq_adj1} suggests that we have $Z_t^{1}=0$, and $Y_t^1=-Q_0-\int_0^t\alpha_sds=-Q_t$; $Y_t^2$ will be determined later.
Substituting in \eqref{FOC_lo}, we have
\begin{eqnarray*}\textstyle
Y_0^2-2\phi\int_0^tQ_sds+\int_0^tZ_s^{2}dW_s+s_0+
\lambda\int_0^t\EE[\alpha_s]ds
+\sigma W_t+2k\alpha_t-\lambda(\EE[Q_0]+\int_0^t\EE[\alpha_s]ds)=0,
\end{eqnarray*}
that is,
\begin{equation}\label{eq_speed}\textstyle
\alpha_t=\frac{\lambda\EE[Q_0]-Y_0^2-s_0}{2k}+\frac{\phi}{k}\int_0^tQ_sds-\frac{1}{2k}\int_0^t(Z_s^{2}+\sigma)dW_s.
\end{equation}

We now show that $Q\equiv Q^0$ and $\alpha\equiv \alpha^0$, where 
\[
Q^0_t:=\EE[Q_t|Q_0],\qquad \alpha^0_t:=\EE[\alpha_t|Q_0].
\]
By taking conditional expectation in \eqref{eq_inv} and \eqref{eq_speed}, we get
\begin{equation}\label{eq_av0}\textstyle
Q_t^0=Q_0+\int_0^t\alpha_s^0ds,\qquad
\alpha^0_t=\alpha_0+\frac{\phi}{k}\int_0^tQ^0_sds.
\end{equation}
Setting $F(t):=Q_t^0$, we note that $F'(t)=\alpha_t^0$, and
$F''(t)=\frac{\phi}{k}F(t)$.
Together with the initial conditions $F(0)=Q_0$ and $F'(0)=\alpha_0$, this gives
\begin{equation}\label{eq_ft}\textstyle
F(t)=\left(\frac{Q_0}{2}-\frac{\alpha_0}{2r}\right)e^{-rt}+\left(\frac{Q_0}{2}+\frac{\alpha_0}{2r}\right)e^{rt},
\end{equation}
where $r=\sqrt{\phi/k}$. 
Now, by taking conditional expectation in equation \eqref{eq_adj2}, and substituting in \eqref{eq_av0}, we obtain
\begin{eqnarray}\label{eq_a0T}
\alpha^0_T&=&\textstyle\frac{\lambda \EE[Q_0]-2AQ_0}{2k}+\frac{\lambda}{2k}\int_0^T\EE[\alpha_t]dt
-\frac{A}{k}\int_0^T\alpha^0_tdt \nonumber\\
&=&\textstyle\frac{\lambda \EE[Q_0]-2AQ_0}{2k}+\frac{\lambda}{2k}(\EE[Q_T]-\EE[Q_0])
-\frac{A}{k}(Q^0_T-Q_0)\\
&=& \textstyle\frac{\lambda}{2k} \EE[Q_T]-\frac{A}{k}Q^0_T,\nonumber
\end{eqnarray}
that is, $F'(T)=\frac{\lambda}{2k} \EE[F(T)]-\frac{A}{k}F(T)$. Imposing this condition, and using \eqref{eq_ft}, we obtain:
\begin{equation}\label{eq_a0}
\alpha_0=Q_0r\frac{d_1e^{-rT}-d_2e^{rT}}{d_1e^{-rT}+d_2e^{rT}} + \frac{\EE[Q_0]4\lambda\phi}{(d_1e^{-rT}+d_2e^{rT})(c_1e^{-rT}+c_2e^{rT})},
\end{equation}
where $d_1=\sqrt{\phi k}-A, d_2=\sqrt{\phi k}+A, c_1=2d_1+\lambda, c_2=2d_2-\lambda$.
From \eqref{eq_speed}, we also have an explicit expression for $Y_0^2=\lambda \EE[Q_0]-s_0-2k\alpha_0$.

Now we use the ansatz: $Z^{2}\equiv-\sigma$, and show that the process
\begin{equation}\label{eq_y2tilde}\textstyle
Y_t^2=Y_0^2-2\phi\int_0^tQ_sds-\sigma W_t
\end{equation}
does satisfy the equation and terminal condition in \eqref{eq_adj2}. Only the latter needs to be shown. First note that, with this ansatz, from \eqref{eq_speed} and \eqref{eq_inv} we have:
\[\textstyle
\alpha_t=\alpha_0+\frac{\phi}{k}\int_0^tQ_sds,\qquad
Q_t=Q_0+\alpha_0t+\frac{\phi}{k}\int_0^t\int_0^sQ_udu\ ds,
\]
thus both processes $\alpha$ and $Q$ are $\sigma(Q_0)-$measurable, that is,
\begin{equation}\label{eq_meas}
Q_t=\EE[Q_t|Q_0]=Q^0_t,\quad \textrm{and}\quad \alpha_t=\EE[\alpha_t|Q_0]=\alpha^0_t.
\end{equation}
We now check that $Y^2$ satisfies the terminal condition in \eqref{eq_adj2}. By \eqref{eq_meas}, \eqref{eq_y2tilde} implies
\[\textstyle
Y^2_T=\lambda \EE[Q_0]-s_0-2k\alpha_0-2\phi\int_0^T Q^0_tdt-\sigma W_T.
\]
On the other hand, by \eqref{eq_meas}, \eqref{eq_a0T} and \eqref{eq_av0},
\begin{eqnarray*}
-S_T+2AQ_T&=&-s_0-\lambda (\EE[Q_T]-\EE[Q_0])-\sigma W_T+2AQ^0_T=-s_0+\lambda \EE[Q_0]-2k\alpha^0_T-\sigma W_T\\
&=&-s_0+\lambda \EE[Q_0]-2k\alpha_0-2\phi\int_0^T Q^0_tdt-\sigma W_T,
\end{eqnarray*}
which yields $Y^2_T=-S_T+2AQ_T$, as wanted. This shows that the process $Z^{2}$ in the ansatz, together with $Y^2$ defined above, do satisfy \eqref{eq_adj2}.
We have seen that this gives $Q_t=F(t)$ and $\alpha_t=F'(t)$, thus from \eqref{eq_ft} we have
\begin{eqnarray*}\textstyle
Q_t=\left(\frac{Q_0}{2}-\frac{\alpha_0}{2r}\right)e^{-rt}+\left(\frac{Q_0}{2}+\frac{\alpha_0}{2r}\right)e^{rt},\;\;
\alpha_t=\left(-\frac{Q_0r}{2}+\frac{\alpha_0}{2}\right)e^{-rt}+\left(\frac{Q_0r}{2}+\frac{\alpha_0}{2}\right)e^{rt},
\end{eqnarray*}
which in turns, by \eqref{eq_a0}, gives
\begin{eqnarray*}
Q_t&=&\textstyle Q_0\frac{d_1e^{-r(T-t)}+d_2e^{r(T-t)}}{d_1e^{-rT}+d_2e^{rT}}+\EE[Q_0]\frac{2\lambda\sqrt{\phi k}(-e^{-rt}+e^{rt})}{(d_1e^{-rT}+d_2e^{rT})(c_1e^{-rT}+c_2e^{rT})}
\\
\alpha_t&=&\textstyle Q_0r\frac{d_1e^{-r(T-t)}-d_2e^{r(T-t)}}{d_1e^{-rT}+d_2e^{rT}}+\EE[Q_0]\frac{2\lambda\phi(e^{-rt}+e^{rt})}{(d_1e^{-rT}+d_2e^{rT})(c_1e^{-rT}+c_2e^{rT})}.
\end{eqnarray*}

\subsection{\textbf{The Linear-Quadratic Case}}
\label{sub:mkv_LQ}
In this subsection, we use the sufficient condition derived above to solve a simple Linear Quadratic (LQ) model. Via different methods, such models have been already studied in the literature; see e.g.\ \cite{BP,PW,yong2013linear,graber2016linear}. For the sake of simplicity, we give the details of the computations in the scalar case $m=d=k=1$ {and with $A=\RR$}. Also as before, we assume that the volatility is not controlled for the sake of simplicity, and in fact that it is identically equal to $1$. In such LQ model, the drift is of the form:
$$
b(x,\alpha,\xi)=b_1 x+b_2\alpha+\bar b_1\bar x + \bar b_2\bar\alpha,
$$
for some constants $b_1,b_2,\bar b_1,\bar b_2$, where we denote by $\bar x$ and $\bar\alpha$ the means of the state and the control in the sense that
$\bar x=\int\int x\xi(dx,d\alpha)$, and $\bar \alpha=\int\int \alpha\xi(dx,d\alpha)$.
As for the cost functions, we assume that:
$$
\textstyle{f(x,\alpha,\xi)=\frac12\left[
q x^2+\bar q (x-s\bar x)^2 + r\alpha^2 +\bar r(\alpha -\bar s\bar \alpha)^2
\right],\quad g(x,\mu)=\frac12\gamma x^2 + \frac{\bar\gamma}2(x-\rho \bar x)^2},
$$
for some constants $q,\bar q,r, \bar r, s,\bar s,\gamma,\delta,\rho$ satisfying $\bar q,\bar r,\bar\gamma\ge 0$ and $q,r,\gamma >0$. 
Under these conditions, the Hamiltonian reads:
\begin{equation}\label{eq Ham LQ ex}
\textstyle{H(x,\alpha,\xi,y)=(b_1 x+b_2\alpha+\bar b_1\bar x + \bar b_2\bar\alpha)y+\frac12\left[
q x^2+\bar q (x-s\bar x)^2 + r\alpha^2 +\bar r(\alpha -\bar s\bar \alpha)^2
\right].}
\end{equation}
Accordingly, the adjoint equation reads as:
\begin{equation}
\label{fo:LQ_adjoint}
dY_t=-\Bigl(b_1Y_t +(q+\bar q)X_t + \bar b_1\EE[Y_t] + s\bar q(s-2)\EE[X_t]\Bigr)dt  + Z_t dW_t.
\end{equation}
In the present situation, conditions (i) and (ii) of Theorem~\ref{th:pontryagin} hold, and condition 
\eqref{fo:Pontryagin_sufficient0} of the Pontryagin stochastic maximum principle holds if:
\begin{equation}
\label{fo:LQ_sufficient}
b_2 Y_t +\bar b_2\EE[Y_t] +(r+\bar r)\alpha_t +\bar r\bar s(\bar s -2) \EE[\alpha_t]= 0.
\end{equation}
Taking expectations, we get
\begin{equation}
\label{fo:bar_alpha_t}\textstyle
\EE[\alpha_t] = -\frac{b_2+\bar b_2}{r+\bar r(\bar s-1)^2}\EE[Y_t].
\end{equation}
Plugging this expression into \eqref{fo:LQ_sufficient}, we get:
\begin{equation}
\label{fo:alpha_t}\textstyle
\alpha_t=-\frac{1}{r+\bar r}\Bigl(b_2 Y_t +\Bigl( \bar b_2-\frac{\bar r \bar s (\bar s -2)(b_2+\bar b_2)}{r+\bar r(\bar s-1)^2}\Bigr) \bar Y_t \Bigr).
\end{equation}
We can rewrite \eqref{fo:alpha_t} and \eqref{fo:bar_alpha_t} as
\begin{equation}
\alpha_t= a Y_t + b \EE[Y_t]\qquad \text{and}\qquad \EE[\alpha_t]=c\EE[Y_t],
\end{equation}
with
\begin{equation}\label{eq:abc}\textstyle
a= -\frac{b_2}{r+\bar r},\quad b= -\frac{1}{r+\bar r}\Bigl( \bar b_2-\frac{\bar r \bar s (\bar s -2)(b_2+\bar b_2)}{r+\bar r(\bar s-1)^2}\Bigr),\quad \text{and}\quad c=-\frac{b_2+\bar b_2}{r+\bar r(\bar s-1)^2}.
\end{equation}
With this notation, the solution of the mean field optimal control of the McKean-Vlasov SDE \eqref{fo:state} reduces to the solution of the following forward-backward stochastic differential equation (FBSDE) of McKean-Vlasov type:
\begin{equation}
\label{fo:MKV_fbsde}
\begin{cases}
&dX_t=\bigl( b_1 X_t +\bar b_1 \EE[X_t] +(ab_2Y_t+{ (bb_2+c\bar b_2)\EE[Y_t]}\Bigr)dt + dW_t\\
&dY_t=-\Bigl(b_1Y_t +(q+\bar q)X_t + \bar b_1\EE[Y_t] + s\bar{q}(s-2)\EE[X_t]\Bigr)dt  + Z_t dW_t,
\end{cases}
\end{equation} 
with terminal condition $Y_T=(\gamma +\bar\gamma) X_T +\bar\gamma\rho(\rho-2)\EE[X_T]$. We solve this system in the usual way. First we compute the means $\bar x_t=\EE[X_t]$ and $\bar y_t=\EE[Y_t]$. Taking expectations in \eqref{fo:MKV_fbsde} we obtain:
\begin{equation}
\label{fo:mean_MKV_fbsde}
\begin{cases}
&d\bar x_t=\bigl( (b_1  +\bar b_1) {\bar x_t} +(ab_2+{ bb_2+c\bar b_2})\bar y_t\Bigr)dt\\
&d\bar y_t=-\Bigl((b_1+\bar b_1)\bar y_t +(q+\bar q + s\bar q(s-2))\bar x_t\Bigr)dt,
\end{cases}
\end{equation} 
with terminal condition $\bar y_T=(\gamma +\bar\gamma+\bar\gamma\rho(\rho-2))\bar x_T$. We search for a solution in the form
{$\bar y_t=\bar\eta_t \bar x_t$, for a deterministic function $t\mapsto \bar\eta_t$}. Computing the derivatives of $\bar y_t$ from the ansatz and from the system \eqref{fo:mean_MKV_fbsde}, and identifying the two, we find that $\bar\eta_t$ should be the solution of the scalar Riccati equation:
\begin{equation}\label{eq:Ricc}
\bar\eta'_t=2A\bar\eta_t+B\bar\eta^2_t+C,
\end{equation}
with terminal condition $\bar\eta_T=K$, where 
$
A=-(b_1  +\bar b_1),  B=-(ab_2+ bb_2+c\bar b_2),  C=-(q+\bar q + s\bar q(s-2))$, and $ K=(\gamma +\bar\gamma+\bar\gamma\rho(\rho-2))$.
By setting $\bar\eta_t=-\frac{\dot{z}_t}{Bz_t}$, solving the above Riccati equation is equivalent to solving the second order linear equation
$
\ddot{z}_t-2A\dot{z}_t+BCz_t=0,
$
with $z_T=1, \dot{z}_T=-BK$. This has explicit solution, whose form depends on the sign of $A^2-BC$. In this way we find the expression for $\bar\eta$, which plugged into \eqref{fo:mean_MKV_fbsde}, and then into \eqref{fo:MKV_fbsde}, reduces the latter to a standard affine FBSDE. This can be solved by noticing that its decoupling field has to be affine, in other words, by searching for a solution in the form $Y_t=\eta_t X_t +\chi_t$ for two deterministic functions $t\mapsto \eta_t$ and $t\mapsto \chi_t$. Computing the stochastic It\^o derivatives of $Y_t$ from the ansatz and from the system \eqref{fo:MKV_fbsde}, and identifying the two, we find that $\eta_t$ should be the solution of the scalar Riccati equation:
\[
\eta_t=-\frac{1}{2b_1}(q+\bar q+\eta'_t+ab_2\eta^2_t),
\]
and once the latter is solved, $\chi_t$ can also be explicitly obtained.
By Theorem~\ref{th:pontryagin}, the control $\balpha$ obtained in this way is optimal. Notice that it takes the form
\[
\alpha_t=a\eta_tX_t+a\chi_t+b\bar\eta_t\EE[X_t],
\]
with $a$ and $b$ given in \eqref{eq:abc}, and $\bar\eta,\eta,\chi$ obtained by solving the above Riccati equations.

\begin{rem}
\label{rem example} 
In classical control of Mean Field type, the pointwise minimization of the Hamiltonian with respect to the control is a necessary optimality condition. Let us illustrate with the LQ example how this need not be the case in our extended framework. If we impose pointwise minimization of \eqref{eq Ham LQ ex} with respect to $\alpha$, we get 
$b_2 Y_t +r\alpha_t +\bar r(\alpha_t -\bar s\bar\alpha_t) =0$.
Integrating it, we obtain $b_2 \EE[Y_t] +(r+\bar{r}-\bar{r}\bar{s})\bar{\alpha_t}  =0$. On the other hand, the necessary condition \eqref{fo:Pontryagin_sufficient} implies \eqref{fo:LQ_sufficient}, so we have $\bar{b}_2\EE[Y_t]+\bar{r}\bar{s}(\bar{s}-1)\bar{\alpha}_t =0$. The right choice of parameters leads to a contradiction between this and the previous equation. 
\end{rem}

\section{Variational Perspective in the Weak Formulation}
\label{sect:wMKV}
The goal of this section is to analyse the extended Mean Field control problem from a purely variational perspective. For this reason, we start by introducing in Section \ref{Sec:subsec weak} a weak formulation of the extended Mean Field control problem, especially well-suited for variational analysis. In such formulation, the probability space is not specified a priori. We remark that a weak formulation of the Mean Field control problem has been considered in \cite[Sect.\ 6.6]{CarmonaDelarue_book_vol_I} and in \cite{La17}, the latter rigorously proving convergence of large systems of interacting control problems to the corresponding Mean Field control problem. However, in these works there is no non-linear dependence on the law of the control; cf.\ our problem \eqref{eq:MKV} below.

We proceed in Section \ref{sec:MKVm} to obtain what we call a \emph{martingale optimality condition}. Such a condition can serve as a verification tool, in order to evaluate whether a given control can be optimal at all. It is therefore the weak-formulation analogue of the necessary Pontryagin maximum principle. {This forms a bridge between the previous sections of this work, and the ensuing ones.} Whenever the Pontryagin maximum principle can be used (or the martingale optimality condition in the weak formulation), it is a powerful tool to identify optimal controls and the trajectories of the state at the optimum. However, it does not say much about the optimal value of the problem. In fact, at the optimum, the adjoint process gives formally the value of the gradient of the value function when computed along the optimal trajectories. In order to study the value function of the control problem (in a situation in which PDE techniques are highly non-trivial) we recast  in Section \ref{Sec:transp} our weak formulation in transport-theoretic terms. {Numerical optimal transport has spectacularly grown in strength over the last few years; see e.g.\ \cite{cuturi2013sinkhorn,benamou2015iterative,2018-Peyre-computational-ot} and the references therein. {Our connection between transport and Mean Field control is meant to lay ground for efficient numerical methods in the future.} In Section \ref{subsec transport discretization} we provide, at a theorerical level, a first discretization scheme of this kind. To be specific, the optimal transport problem we obtain in the discretization has an additional \emph{causality} constraint (see e.g.\ \cite{La,ABZ,BBEP,BBLZ}}); the numerical analysis of such problems is also having a burst of activity (e.g.\ \cite{pflug2009version,pflug2012distance,pflug2016multistage}).

\subsection{The Weak Formulation}
\label{Sec:subsec weak}

We present a \emph{weak formulation} of the extended Mean Field control problem formulated in Section~\ref{sect:gMKV}, in the sense that the probability space is not specified here. 
We restrict our attention to the case where the state dynamics have uncontrolled volatility, {actually assuming $\sigma\equiv I_d$}, $m=d$, that the drift does not depend on the law of the control, and that the initial condition $X_0$ is a constant $x_0$. We thus consider the minimization problem
\begin{eqnarray}\label{eq:MKV}
\begin{split}
&&\inf_{\PP,\balpha}{\mathbb E}^\PP\left[\int_0^T f(X_t,\alpha_t,\cL_\PP(X_t,\alpha_t))dt+g(X_T,\cL_\PP(X_T))\right]\\
&& \textrm{subject to}\quad dX_t=b\left(X_t,\alpha_t,\cL_\PP(X_t)\right)dt+dW_t,\,\, X_0=x_0,
\end{split}
\end{eqnarray}
where the infimum is taken over filtered probability spaces $(\Omega,\FF,\PP)$ supporting some d-dimensional Wiener process $\bW$, and  over control processes $\balpha$ which are progressively measurable on $(\Omega,\FF,\PP)$ and $\RR^k$-valued. We use $\cL_\PP$ to denote the law of the given random element under $\PP$. Again, we choose time independent coefficients for simplicity, but all the results would be the same should $f$ and $b$ depend upon $t$.

We say that $(\Omega,\FF,\PP,\bW,\bX,\balpha)$ is a feasible tuple if it participates in the above optimization problem yielding a finite cost.

\subsection{Martingale Optimality Condition}
\label{sec:MKVm}
We now obtain a necessary Pontryagin principle for the weak formulation \eqref{eq:MKV}. We call this the \emph{martingale optimality condition}. Since our aim is to illustrate the method, we assume only in this part that we are dealing with a drift-control problem $$b(x,\alpha,\mu) = \alpha.$$
We start by expressing the objective function of \eqref{eq:MKV} in canonical space, as a function of semimartingale laws. We denote by $\Ccal_{x_0}$ the space of $\RR^d$-valued continuous paths started at $x_0$, and $\bS$ the canonical process on it. We consider the set of semimartingale laws
\begin{equation}\label{eq set S}
\tilde \Pcal:=\{\bmu\in\Pcal(\Ccal_{x_0}) : dS_t=\alpha_t^{\bmu}(\bS) dt+dW_t^{\bmu}\; \mu\textrm{-a.s.}\},
\end{equation}
where $\bW^{\bmu}$ is a $\bmu$-Brownian motion and $\balpha^{\bmu}$ is a progressively measurable process w.r.t.\ the canonical filtration, denoted by $\Fcal$. 
It is then easy to see that \eqref{eq:MKV} is equivalent to 
\begin{equation}\label{eq:MKVS}
\inf_{ \bmu\in \tilde \Pcal} \mathbb E^{\bmu}\left[\int_0^T f\big(S_t,\alpha^{\bmu}_t,\cL_{\bmu}(S_t,\alpha^{\bmu}_t)\big)dt+g(S_T,\mu_T)\right].
\end{equation}
In what follows we consider perturbation of measures in $\tilde \Pcal$ via push-forwards along absolutely continuous shifts which preserve the filtration; see the work of Cruzeiro and Lassalle \cite{CL} and the references therein. {Using push-forwards instead of perturbations directly on the SDE is the main difference between the weak and the strong perspective.}

\begin{rem}\label{rem:K}
Let $\bmu\in\tilde \Pcal$. We say that an adapted process $U:\Ccal_{x_0}\to\Ccal_{x_0}$ is $\bmu$-invertible, if there exists $V:\Ccal_{x_0}\to\Ccal_{x_0}$ adapted such that $U\circ V =\mathrm{id}_{\Ccal_{x_0}}$ holds $U(\bmu)-$a.s., and  $V\circ U = \mathrm{id}_{\Ccal_{x_0}}$ holds $\bmu-$a.s.. Now let $K_\cdot=\int_0^.k_tdt$ be adapted. We say that $K$ preserves the filtration under $\bmu$, if for every $U$ which is $\bmu$-invertible we also have that $U+K$ is $\bmu$-invertible. It follows that the set of those $K=\int_0^.k_tdt$ that preserve the filtration under $\bmu$, is a linear space. It also follows that for such $K$ we have $\bmu^{\epsilon,K}:=(Id+\epsilon K)_*\bmu\in\tilde \Pcal$, with $\alpha_t^{\bmu^{\epsilon,K}}(\bS+\epsilon K(\bS))=\alpha_t^{\bmu}(\bS)+\epsilon k_t(\bS)$; see \cite[Proposition 2.1, Lemma~3.1]{CL}.
\end{rem}

In analogy to Theorem~5.1 in \cite{CL}, we then obtain the following necessary condition for an optimizer in \eqref{eq:MKVS}. We use here the  notation $\theta^{\bmu}_t=(S_t,\alpha^{\bmu}_t,\cL_{\bmu}(S_t,\alpha^{\bmu}_t))$.
\begin{prop}
Let $\bmu$ be an optimizer for \eqref{eq:MKVS}. Then the process  $N^{\bmu}$ given by
\begin{equation}\label{eq:mg}
N_t^{\bmu}:=
\partial_a f(\theta^{\bmu}_t)
+\tilde{\mathbb E}[
\partial_\nu f(\theta^{\bmu}_t)(\tilde S_t,\tilde\alpha_t)]
-\int_0^t\Big(\partial_x f(\theta^{\bmu}_s)
+\tilde{\mathbb E}[
\partial_\mu f(\theta^{\bmu}_s)(\tilde S_s,\tilde\alpha_s)]\Big)ds 
\end{equation}
is a $\bmu$-martingale with terminal value equal to
\begin{equation}\label{eq:tv}
N_T^{\bmu}=-\partial_x g(S_T,\mu_T)
-\tilde{\mathbb E}[\partial_\mu g(S_T,\mu_T)(\tilde S_T)]
-\int_0^T\Big(\partial_x f(\theta^{\bmu}_s)
+
\tilde{\mathbb E}[
\partial_\mu f(\theta^{\bmu}_s)(\tilde S_s,\tilde\alpha_s)]\Big)ds.
\end{equation}
\end{prop}

\begin{proof}
{We use the notation $\mu^{\epsilon,K}$ introduced in Remark \ref{rem:K}. Call $C(\mu)$ the cost function appearing in Problem \eqref{eq:MKVS}. We have $\lim_{\epsilon\to 0}\frac{C(\mu^{\epsilon,K})-C(\mu)}{\epsilon}\geq 0$ for all $K$. Now if $K$ preserves the filtration under $\bmu$, then the same is true for $-K$. Therefore $\lim_{\epsilon\to 0}\frac{C(\mu^{\epsilon,K})-C(\mu)}{\epsilon}= 0$.} To conclude the proof, we use $\alpha_t^{\bmu^{\epsilon,K}}(\bS+\epsilon K(\bS))=\alpha_t^{\bmu}(\bS)+\epsilon k_t(\bS)$ and similar arguments as in Theorem~5.1 in \cite{CL}.
\end{proof}

When \eqref{eq:mg}-\eqref{eq:tv} hold, we say that $\mu$ satisfies the \emph{martingale optimality condition}. The interest of this condition is that it is a clear stochastic counterpart to the classical Euler-Lagrange condition in calculus of variation, except for the fact that ``being equal to zero'' is replaced by ``being a martingale''; see \cite{LZ,CL}. 

\subsection{Optimal Transport Reformulation}
\label{Sec:transp}
In this section we formulate a variational transport problem on $\Ccal=\Ccal([0,T];\RR^d)$, the space of $\RR^d$-valued continuous paths, which is equivalent to finding the weak solutions of the extended Mean Field problem \eqref{eq:MKV}.  This variational formulation is a particular type of transport problem under the so-called \emph{causality} constraint{; see \cite{La,ABZ,BBEP,BBLZ}}. Here we recall this concept with respect to
the filtrations $\FF^1$ and $\FF^2$, generated by the first and by the second coordinate process on $\Ccal\times\Ccal$.
\begin{defn}\label{def causal}
Given $\bzeta_1,\bzeta_2\in\Pcal(\Ccal)$, a probability measure $\pi\in\Pcal(\Ccal\times\Ccal)$ is called a causal transport plan between $\bzeta_1$ and $\bzeta_2$ if its marginals are  $\bzeta_1$ and $\bzeta_2$, and, for any $t\in [0,T]$ and any set $A \in \Fcal^2_t$, the map $ \Ccal\ni x \mapsto \pi^x(A)
$ is $\tilde\Fcal^1_t$- measurable, where
$\pi^x(dy):=\pi(\{x\}\times dy)$ is a regular conditional kernel of $\pi$ w.r.t.\ the first coordinate, and $\tilde\FF^1$ is the completion of $\FF^1$ w.r.t.\ $\bzeta_1$. The set of causal transport plans between $\bzeta_1$ and $\bzeta_2$ is denoted by $\Pi_{c}(\bzeta_1,\bzeta_2)$.
\end{defn}

The only transport plans that contribute to the variational formulation of the problem are those under which the difference of the the coordinate processes on the product space $\Ccal\times\Ccal$ is a.s.\ absolutely continuous with respect to Lebesgue measure.
We denote by $(\omega,\sw)$ the generic element on  $\Ccal\times\Ccal$, and we use $(\dot{\wideparen{\sw-\omega}})$ to indicate the density of the process $\sw-\omega$ with respect to Lebesgue measure when it exists, i.e. $$\sw_t-\omega_t=\sw_0-\omega_0+\int_0^t(\dot{\wideparen{\sw-\omega}})_s\, ds,\,\, t\in[0,T].$$  
In such case we write $\sw-\omega\ll\mathcal L$.
Moreover, we set $$\bgamma\,:=\, \text{ Wiener measure on $\Ccal$ started at 0},$$ and
$\Pi_c^{\ll}(\gamma,\cdot):=\left\{\pi\in\mathcal P(\Ccal\times\Ccal):\, \pi(d\omega\times \Ccal)=\gamma(d\omega),\text{ and }\,\,\sw-\omega\ll\mathcal L,\, \pi\text{-a.s.}   \right\}$. 

We present the connection between extended Mean Field control and causal transport:
\begin{lem}\label{lm:CTMKV}Assume that $b(x,.,\mu)$ is injective, and set $$u_t(\omega,\sw,\mu):=b^{-1}(\sw_t,.,\mu)\big(
(\dot{\wideparen{\sw-\omega}})_t\big).$$ Then Problem \eqref{eq:MKV} is equivalent to:
\begin{eqnarray}\label{eq:VPMKV} 
\inf 
{\mathbb E}^\pi\left[\int_0^Tf\big(\sw_t,u_t(\omega,\sw,\mu_t^\pi),\cL_\pi(\sw_t,u_t(\omega,\sw,\mu_t^\pi))\big)dt+g(\sw_T,\mu^\pi_T)
\right],
\end{eqnarray}
where the infimum is taken over transport plans $\pi\in \Pi_c^{\ll}(\gamma,\cdot)$ such that $\text{Leb}_1\otimes\pi$-a.s.\ $ (\dot{\wideparen{\sw-\omega}} )_t \in b(\sw_t,\RR^d,\mu^\pi_t)$, and $\mu^\pi$ denotes the second marginal of $\pi$. 
\end{lem}

\begin{proof}
Fix $(\Omega,\FF,\PP,\bW,\bX,\balpha)$  feasible tuple for \eqref{eq:MKV}, if it exists, and note that $\alpha_t=u_t(\bW,\bX,\cL_\PP(X_t))$ is $\FF^{\bX,\bW}$-adapted.
Then $\pi:=\cL_\PP(\bW,\bX)$ belongs to $\Pi_c^\ll(\bgamma,\cL_\PP(\bX))$ and generates the same cost in \eqref{eq:VPMKV}. Conversely, given a transport plan $\pi$ participating in \eqref{eq:VPMKV}, the following tuple $(\Omega,\FF,\PP,\bW,\bX,\balpha)$  is feasible for \eqref{eq:MKV}: $\Omega=\Ccal\times\Ccal$, $\FF$ canonical filtration on $\Ccal\times\Ccal$, $\PP=\pi$, $\bW=\omega$, $\bX=\sw$, and $\alpha_t=u_t(\omega,\sw,\mu^\pi_t)$.
\end{proof}

We illustrate in the next proposition the connection presented in the above lemma. We say that a feasible tuple for \eqref{eq:MKV} is a \textit{weak closed loop} if the control is adapted to the state (i.e.\ $\balpha$ is $\FF^{\bX}$-measurable).

\begin{prop}\label{prop wcl}
Assume
\begin{itemize}
\item[(A1)] $b(x,.,\mu)$ is injective, $b(x,\RR^k,\mu)$ is a convex set, and $b^{-1}(x,.,\mu)$ is convex;
\item[(A2)]  $f(x,b^{-1}(x,.,\mu),\xi)$ is convex {and grows at least like $\kappa_0+\kappa_1|\cdot|^p$ with $\kappa_1>0,p\geq 1$};
\item[(A3)] $f(x,\alpha,.)$ is $\prec_{cm}$-monotone.\footnote{A function $f:\Pcal(\RR^N)\to\RR$ is called $\prec_{cm}$-monotone (resp. $\prec_{c}$-monotone) if $f(m_1)\leq f(m_2)$ whenever $m_1\prec_{cm}m_2$ (resp. $m_1\prec_{c}m_2$). With the latter order of measures, we mean  $\int h dm_1\leq \int h dm_2$ for all functions $h$ which are convex and increasing w.r.t.\ the usual componentwise order in $\RR^N$ (resp.\ all convex functions $h$)  such that the integrals exist.}
\end{itemize}
Then the minimization in the extended Mean Field problem \eqref{eq:MKV}  can be taken over weak closed loop tuples. Moreover, if the infimum is attained, then the optimal control $\balpha$ is of weak closed loop form. 
\end{prop}

The proof follows the projection arguments used in \cite{ABZ}, which requires the above convexity assumptions. On the other hand, no regularity conditions are required here, unlike in the classical PDE or probabilistic approaches (see Assumptions (I)-(II) in Section \ref{sect:Pontryagin}). We refer to \cite{La17} for a similar statement, in a general framework, but under no non-linear dependence on the control law. This proof is postponed to Appendix \ref{sec app proof}.

\begin{rem}
If $b$ is linear with positive coefficient for $\alpha$, then (A3) can be weakened:
\begin{itemize}
\item[(A3')] $f(x,\alpha,.)$ is $\prec_{c}$-monotone,
\end{itemize} 
as can be seen from the proof. For example, conditions (A1),(A2),(A3') are satisfied if
\[
b(x,\alpha,\mu)=c_1x+c_2\alpha+c_3\bar\mu,\qquad f(x,\alpha,\xi)=d_1x+d_2\alpha+d_3x^2+d_4\alpha^2+J(\bar\xi_1,\bar\xi_2),
\]
where $J$ is a measurable function, $\bar\mu=\int x\mu(dx)$, $\bar\xi_1=\int\int x\xi(dx,d\alpha)$, $\bar\xi_2=\int\int \alpha\xi(dx,d\alpha)$, and $c_i,d_i$ are constants such that $c_2\neq 0$, $d_4/c_2>0$.
\end{rem}

\subsection{A Transport-Theoretic Discretization Scheme}
\label{subsec transport discretization}

In this part we specialize the analysis to the following particular case of \eqref{eq:MKV}:
\begin{eqnarray}\label{eq:MKV special}
\begin{split}
&&\inf_{\PP,\balpha}\left\{\int_0^1 f(\cL_\PP(\alpha_t))dt+g(\cL_\PP(X_T)): \,\,\,\, dX_t=
\alpha_tdt+dW_t\, ,\, X_0=x_0\right\},
\end{split}
\end{eqnarray}
where for simplicity we took $T=1$.  Throughout this section we assume:

\begin{enumerate}
\item $g$ is bounded from below and {lower semicontinuous w.r.t.\ weak convergence};
\item $f$ is increasing with respect to convex order, lower semicontinuous w.r.t.\ weak convergence, and such that for all
$\lambda\in [0,1]$ and $\R^N$-valued random variables $Z,\bar{Z} $:
\begin{align}\label{eq convex in law}
f(\cL(\lambda Z + (1-\lambda) \bar{Z}))\leq \lambda f(\cL(Z)) +(1-\lambda) f(\cL(\bar{Z}));
\end{align}
\item $f$ satisfies the growth condition $f(\rho)\geq a + b \int |z|^pd\rho(z) $ for some $a\in\R,b>0,p>1$.
\end{enumerate}

Lemma \ref{lm:CTMKV} shows the equivalence of \eqref{eq:MKV special} with the variational problem
\begin{align*}   \textstyle
\inf_{\pi\in\Pi_{c}^\ll(\bgamma,\cdot)}
\left \{ \int_0^1 f\big(\cL_\pi( \dot{\wideparen{\sw-\omega}}  )_t\big)dt+g(\cL_\pi(\sw_T))
\right\}.
\end{align*}
Under the convention that $\int_0^1 f\big(\cL_\pi( \dot{\wideparen{\sw-\omega}}  )_t\big)dt=+\infty$ if $\sw-\omega \ll\cL$ fails under $\pi$, the latter can be expressed in the equivalent form:
\begin{align}\label{eq:VPMKV special}\tag{$P$}
\inf_{\bmu\in\tilde\Pcal}\,   
\inf_{\pi\in\Pi_{c}(\bgamma,\bmu)}
\left \{ \int_0^1 f\big(\cL_\pi( \dot{\wideparen{\sw-\omega}}  )_t\big)dt+g(\cL_\pi(\sw_T))
\right\},
\end{align}
where $\tilde\Pcal$ was defined in \eqref{eq set S}. In the same spirit as \cite[Ch.\ 3.6]{Z}, we introduce a family of causal transport problems in finite dimension increasing to \eqref{eq:VPMKV special}. For $n\in \mathbb{N}$, let $\mathbb{T}_n:=\{i\,2^{-n}\,:\, 0\leq i\leq 2^n,\, i\in\mathbb{N}\}$ be the $n$-th generation dyadic grid. For measures $m\in\mathcal{P}(\Ccal)$ and $\pi\in\Pcal(\Ccal\times \Ccal)$, we write
$$m_n := \cL_m(\{\omega_t\}_{t \in \mathbb{T}_n})\in\Pcal(\RR^{ (2^n+1) d}) \,\, \text{ and }\,\,\pi_n:=  \cL_\pi(\{(\omega_t,\sw_t)\}_{t \in \mathbb{T}_n})\in\Pcal(\RR^{ (2^n+1) d}\times \RR^{ (2^n+1) d})  ,$$
the projections of $m$ and $\pi$ on the grid $\mathbb{T}_n$. 
We denote by $(x_0^n,x_1^n,\dots,x_{2^n}^n,y_0^n,y_1^n,\dots,y_{2^n}^n)$ a typical element of $\RR^{ (2^n+1) d}\times \RR^{ (2^n+1) d}$, and let $\Delta^n x_i :=x_{i+1}^n-x^n_i$, and similarly for $\Delta^n y_i$.

We consider the auxiliary transport problems
\begin{align}\tag{$P(n)$}\label{eq:VPMKV nth}
\inf_{\bmu\in\Pcal(\RR^{ (2^n+1) d}) }\,   
\inf_{\pi\in\Pi_{c}^n(\bgamma_n,\bmu)}
\left \{ 2^{-n} \sum_{i=0}^{2^n-1} f\left(\cL_\pi\left ( \frac{ \Delta^n y_i  -  \Delta^n x_i}{2^{-n}} \right )\right) +g(\cL_\pi(y_{2^n}^n))
\right\},
\end{align}
where, in analogy to Definition \ref{def causal}, we called
$$\Pi_{c}^n(\bgamma_n,\bmu)\subset \Pcal(\RR^{ (2^n+1) d}\times \RR^{ (2^n+1) d})$$  the set of causal couplings in $\Pcal(\RR^{ (2^n+1) d}\times \RR^{ (2^n+1) d})$ with marginals $\bgamma_n$ and $\bmu$; see \cite{BBLZ}.

\begin{thm}
\label{thm approx}
Suppose Problem \eqref{eq:VPMKV special} is finite, and that (i),(ii),(iii) hold. Then the value of the auxiliary problems \eqref{eq:VPMKV nth} increases to the value of the original problem \eqref{eq:VPMKV special}, {and the latter admits an optimizer.}
\end{thm}

\begin{rem}
An example of a function satisfying Conditions (ii)-(iii) of Theorem \ref{thm approx} is
$f(\rho)= R\left(\int hd\rho\right)$, for $R$ convex and increasing, and $h$ convex with $p$-power growth ($p>1$). It also covers the case of functions of the form
 $f(\rho)= \int \phi(w,z)d\rho(w)d\rho(z) + \int |x|^p d\rho(x)$, with $\phi$ jointly convex and bounded from below, and $f(\rho) = \mathrm{Var}(\rho)+\int |x|^p d\rho(x)$,
 where in both cases $p>1$. For $p=2$ the latter falls into the LQ case of Section \ref{sub:mkv_LQ}.
\end{rem}

\begin{proof}
{\bf Step 1} (Lower bound). Let $\bmu \in \tilde{\Pcal}$ and $\pi\in\Pi_c(\bgamma,\bmu)$ with finite cost for Problem \eqref{eq:VPMKV special}, $n\in\mathbb{N}$,
 and $\pi_n$ be the projection of $\pi$ onto the grid $\mathbb{T}_n$. We first observe that
\begin{align}
\label{eq low bound}\int_0^1 f\big(\cL_\pi( \dot{\wideparen{\sw-\omega}}  )_t\big)dt+g(\cL_\pi(\sw_T)) \geq  2^{-n} \sum_{i=0}^{2^n-1} f\left(\cL_{\pi_n}\left ( \frac{ \Delta^n y_i - \Delta^n x_i}{2^{-n}} \right )\right) +g(\cL_{\pi_n}(y^n_{2^n})).
\end{align}
Indeed, for $i\in\{0,\dots,2^n-1\}$ we have
\begin{align*}
\int_{i2^{-n}}^{(i+1)2^{-n}} f\big(\cL_\pi( \dot{\wideparen{\sw-\omega}}  )_t\big)dt 
&\geq 2^{-n} f\left(\cL_\pi\left(\int_{i2^{-n}}^{(i+1)2^{-n}}( \dot{\wideparen{\sw-\omega}} )_t \,\frac{dt}{2^{-n}}\right)\right)\\
& =  2^{-n} f\left(\cL_\pi\left( \frac{\sw_{(i+1)2^{-n}}   - \sw_{i2^{-n}}-(\omega_{(i+1)2^{-n}}   - \omega_{i2^{-n}}) }{2^{-n}}\right)\right)\\
& = 2^{-n} f\left(\cL_{\pi_n}\left( \frac{\Delta^n y_i - \Delta^n x_i }{2^{-n}}\right)\right),
\end{align*}
where for the inequality we used the convexity condition \eqref{eq convex in law}. Noticing that the first marginal of $\pi_n$ is equal to $\bgamma_n$, the r.h.s.\ of \eqref{eq low bound} is bounded from below by the value of \eqref{eq:VPMKV nth}. Because $\bmu,\pi$ have been chosen having finite cost for Problem \eqref{eq:VPMKV special}, but otherwise arbitrary, we conclude that
\begin{align*}
\eqref{eq:VPMKV special}\geq \eqref{eq:VPMKV nth}\quad \forall n\in\NN.
\end{align*}

\noindent{\bf Step 2} (Monotonicity). For $n\in\mathbb{N}$ and $i\in\{0,\dots,2^n-1\}$, take $k$ such that
$$i 2^{-n} = (k-1)2^{-(n+1)}<k2^{-(n+1)}< (k+1)2^{-(n+1)}= (i+1) 2^{-n}.$$
Let $\bmu_{n+1}\in\Pcal(\RR^{ (2^{n+1}+1) d}) $ and $\pi_{n+1}\in\Pi_{c}^{n+1}(\bgamma_{n+1},\bmu_{n+1})$. By \eqref{eq convex in law} we get
\begin{align*}& \textstyle
2^{-(n+1)}\left\{ f\left(\cL_{\pi_{n+1}}\left ( \frac{ \Delta^{n+1} y_{k-1}  - \Delta^{n+1} x_{k-1} }{2^{-(n+1)}} \right )\right) + f\left(\cL_{\pi_{n+1}}\left (  \frac{\Delta^{n+1} y_{k}  - \Delta^{n+1} x_{k} }{2^{-(n+1)}} \right )\right)
\right \} \\
\geq & \textstyle \, 2^{-n}f\left(\cL_{\pi_{n+1}}\left (\frac{ y_{k+1}^{n+1}-y_{k-1}^{n+1}-(x_{k+1}^{n+1}-x_{k-1}^{n+1})}{2^{-n}} \right )\right) \, = \, 2^{-n}f\left(\cL_{\pi_{n}}\left ( \frac{\Delta^n y_{i}-\Delta^n x_{i}}{2^{-n}} \right )\right),
\end{align*}
where $\pi_n$ is the projection of $\pi_{n+1}$ on the grid $\mathbb{T}_n$. We conclude as in the previous step
\begin{align*}
(P(n+1))\geq \eqref{eq:VPMKV nth}\quad \forall n\in\NN.
\end{align*}

\noindent{\bf Step 3} (Discrete to Continuous). We introduce the auxiliary problems in path-space
\begin{align}\label{eq:VPMKV special aux}\tag{$P^{aux}(n)$}\textstyle
\inf_{\bmu\in\tilde\Pcal}\,   
\inf_{\pi\in\Pi_{c}(\bgamma,\bmu)}
\left \{ 2^{-n} \sum_{i=0}^{2^n-1} f\left(\cL_\pi\left ( \frac{ \Delta^n_i \sw  - \Delta^n_i \omega}{2^{-n}} \right )\right) +g(\cL_\pi(\sw_1))
\right\},
\end{align}
where $\Delta^n_i \omega:= \omega_{(i+1)2^{-n}} -\omega_{i2^{-n}} $ and likewise for $\Delta^n_i \sw$. We now prove that 
\begin{align}\label{eq:equality discrete to continuous}
\eqref{eq:VPMKV special aux} = \eqref{eq:VPMKV nth}\quad \forall n\in\NN.
\end{align}
First we observe that the l.h.s.\ of \eqref{eq:equality discrete to continuous} is larger than the r.h.s. Indeed, projecting a coupling from $\Pi_c(\bgamma,\cdot)$ onto a discretization grid gives again a causal coupling; see \cite[Lemma 3.5.1]{Z}. 
The converse inequality follows by Remark~\ref{rem SDE}, since it implies in particular that for any $\bnu\in\Pcal(\RR^{ (2^n+1) d})$ and $\pi\in\Pi_{c}^n(\bgamma_n,\bnu)$ with finite cost in \eqref{eq:VPMKV nth}, there exist $\bmu\in\tilde\Pcal$ and $P\in\Pi_{c}(\bgamma,\bmu)$ giving the same cost in \eqref{eq:VPMKV special aux}.\\
{\bf Step 4} (Convergence). Let us denote  $$ \textstyle c(\pi):= \int_0^1 f\big(\cL_\pi( \dot{\wideparen{\sw-\omega}}  )_t\big)dt\,\, \text{ and }\,\, c^n(\pi):=2^{-n} \sum_{i=0}^{2^n-1} f\left(\cL_\pi\left ( \frac{ \Delta^n_i \sw  - \Delta^n_i \omega}{2^{-n}} \right )\right), $$ the cost functionals defining the optimization problems \eqref{eq:VPMKV special} and \eqref{eq:VPMKV special aux}. Notice that Step 1 implies that $c\geq c^n$, and Step 2 shows that $c^n$ is increasing.  We now show that $c^n$ converges to $c$ whenever the latter is finite. For this it suffices to show
  \begin{align}\label{eq to do Leb diff}
  \liminf_n c^n(\pi)\geq c(\pi). 
  \end{align}
 We start by representing $c^n$ in an alternative manner, namely
 $$ c_n(\pi)=\int_0^1 f\left ( \,\Lcal_\pi \left (\int_{\lfloor t2^n\rfloor 2^{-n}}^{(\lfloor t2^n\rfloor +1) 2^{-n}} (\dot{\wideparen{\sw-\omega}})_s \frac{ds}{2^{-n}} \right )\, \right )  dt.$$ 
 By Lebesgue differentiation theorem, for each pair $(\sw,\omega)$ such that $\sw-\omega$ is absolutely continuous, there exists a $dt$-full set of times such that
 \begin{align}\label{eq leb diff}
A(t,n):= \int_{\lfloor t2^n\rfloor 2^{-n}}^{(\lfloor t2^n\rfloor +1) 2^{-n}} (\dot{\wideparen{\sw-\omega}})_s \frac{ds}{2^{-n}}\rightarrow (\dot{\wideparen{\sw-\omega}})_t . 
 \end{align}
 If $c(\pi)<\infty$, the set of such pairs $(\sw,\omega)$ is $\pi$-full. This shows that \eqref{eq leb diff} holds $\pi(d\omega,d\sw)dt$-a.s. By Fubini theorem, there is a $dt$-full set of times $I\subset [0,1]$ such that, for $t\in I$, the limit \eqref{eq leb diff} holds in the $\pi$-almost sure sense (the $\pi$-null set depends on $t$ a priori). By dominated convergence, this proves that
 $$\textstyle \forall t\in I:\,\, \Lcal_\pi \left (  A(t,n) \right )  \Rightarrow \Lcal_\pi \big ( (\dot{\wideparen{\sw-\omega}})_t  \big),$$
 namely in the sense of weak convergence of measures. By lower-boundedness and lower-semicontinuity of $f$, together with Fatou's Lemma, we obtain
 $$\textstyle \liminf_n c^n(\pi)\geq \int_0^1  \liminf_n f\left (\, \Lcal_\pi \left ( A(t,n)\right )   \, \right )dt = \int_0^1  f\big(\, \Lcal_\pi \big ( (\dot{\wideparen{\sw-\omega}})_t  \big )  \, \big ) dt, $$
 establishing \eqref{eq to do Leb diff} and so that $c^n\nearrow c$.

By Steps 2 and 3, we know that the values of \eqref{eq:VPMKV special aux} are increasing and bounded from above by the value of \eqref{eq:VPMKV special}. We take $\pi^n$ which is $1/n$-optimal for \eqref{eq:VPMKV special aux}. It follows then by Assumptions (i)-(iii) that 
$\int \int_0^1[ (\dot{\wideparen{\sw-\omega}})_t]^p \,dt\,d\pi^n \leq \bar{a}+\bar{b}\eqref{eq:VPMKV special}$,
for some $\bar a, \bar b \in \RR$. 
  By \cite[Lemma 3.6.2]{Z} we obtain the tightness of $\{\pi^n\}_n$. We may thus assume that $\pi_n\Rightarrow \pi$ weakly. By \cite[Lemma 5.5]{ABZ} the measure $\pi$ is causal (and it obviously has first marginal $\gamma$). For $k\leq n$ we have 
  $$ \textstyle c^k(\pi^n) \leq c^n(\pi^n) \leq 1/n + \eqref{eq:VPMKV special aux}\leq 1/n + \eqref{eq:VPMKV special} , $$
  so sending $n\to \infty$ we get $c^k(\pi)\leq \lim_n\eqref{eq:VPMKV special aux}\leq \eqref{eq:VPMKV special}$, as clearly $c^k$ is lower semicontinuous. By letting $k\to \infty$,  and using the fact that $c^k\nearrow c$, we conclude that $\pi$ is optimal for \eqref{eq:VPMKV special} and that the value of \eqref{eq:VPMKV special} is the limit of the increasing values of \eqref{eq:VPMKV special aux}, which in turn equals the  limit of the increasing values of \eqref{eq:VPMKV nth}.
\end{proof}

{We complete the argument used in Step 3 with the following remark. This also shows how, from an (approximate) optimizer of the discrete-time problem \eqref{eq:VPMKV nth}, an approximate optimizer of the continuous-time problem \eqref{eq:MKV special} can be built.
\begin{rem}\label{rem SDE}
In Lemma~\ref{lem SDE} we show how, given $\bnu\in\Pcal(\mathbb{R}^{2d})$ and $\pi\in\Pi_c^0(\bgamma_0,\bnu)$, there exists a weak solution $(W,X)$ of an SDE such that
$\Lcal_P(W_0,W_1,X_0,X_1)=\pi$. The argument used to prove Lemma~\ref{lem SDE} can be iterated in order to get an SDE whose unique weak solution fits any joint distribution over finitely many time-points: For any given $\bnu\in\Pcal(\RR^{ (2^n+1) d})$ and $\pi\in\Pi_{c}^n(\bgamma_n,\bnu)$, there exist $\bmu\in\tilde\Pcal$ and $P\in\Pi_{c}(\bgamma,\bmu)$ such that
$$\Lcal_P(\omega_0,\omega_{2^{-n}},\omega_{2^{-n+1}},\dots,\omega_1,\sw_0,\sw_{2^{-n}},\sw_{2^{-n+1}},\dots,\sw_1)=\pi,$$
with $P$ being the joint law of $(W,X)$, the unique weak solution of an SDE of the form
$$dX_t = \beta_t dt + dW_t.$$
Lemma~\ref{lem SDE} covers the case $n=0$. We now show the case $n=1$, the general case following similarly.
Fix $\bnu\in\Pcal(\RR^{3d})$ and $\pi\in\Pi_{c}^1(\bgamma_1,\bnu)$. As in Lemma~\ref{lem SDE}, if $U_1$ is a $d$-dimensional uniform distribution independent of $X_0$ and the Brownian motion $W$, then there exists $\Psi_1$ such that $(0,W_{1/2},X_0,\Psi_1(U_1,W_{1/2},X_0))\sim \pi_1$, where $\pi_1$ is the projection of $\pi$ into the first $4$ coordinates. Introducing $U_2$ an independent copy of $U_1$, we can apply Lemma \ref{lem selection applica} in the Appendix, obtaining the existence of a measurable function $\Psi_2$ such that
$$\Bigl (\,0,W_{1/2},W_1,X_0,\Psi_1(U_1,W_{1/2},X_0),\Psi_2\bigl (\,U_2,W_{1/2},W_1,X_0,\Psi_1(U_1,W_{1/2},X_0) \,\bigr)\,\Bigr )\sim \pi.$$
Now we define the following SDE with initial condition $X_0$:
$$\textstyle dX_t = \left(\frac{\Psi_1(U_1,W_t,X_0)-X_t}{1/2-t}1_{[0,1/2)}(t)+\frac{\Psi_2(U_2,W_{1/2},W_t,X_0,X_{1/2})-X_t}{1-t}1_{[1/2,1)}(t)\right)dt + dW_t.$$
There is a unique solution in $[0,1)$, which is given by
\begin{align*}
X_t &\textstyle = X_0(1-2t)1_{[0,1/2]}(t)+X_{\frac12}(2-2t)1_{(1/2,1)}(t)\\
&\textstyle + \left(\frac12-t\wedge\frac12\right)\int_0^{t\wedge 1/2} \frac{\Psi_1(U_1,W_s,X_0)}{(1/2-s)^2}ds +(1-t)\int_{t\wedge 1/2}^t \frac{\Psi_2(U_2,W_{1/2},W_s,X_0,X_{1/2})}{(1-s)^2}ds\\
&\textstyle +\left(\frac12-t\wedge\frac12\right)\int_0^{t\wedge 1/2} \frac{1}{1/2-s}dW_s
 +(1-t)\int_{t\wedge 1/2}^t \frac{1}{1-s}dW_s.
\end{align*}
Noting $X_{\frac12-}=\Psi_1(U_1,W_{1/2},X_0)$ and $X_{1-}=\Psi_2(U_2,W_{1/2},W_1,X_0,X_{1/2})$ we conclude.
\end{rem}
}

\appendix

\section{Proof of Proposition \ref{prop wcl}}
\label{sec app proof}

\begin{proof}
Fix $(\Omega,\FF,\PP,\bW,\bX,\balpha)$  feasible tuple for \eqref{eq:MKV}, if it exists, and set $\pi:=\cL_\PP(\bW,\bX)\in\Pi_c^\ll(\bgamma,\cdot)$ and $\bmu:=\bmu^\pi$. Under $\pi$ we have $\sw_t-\omega_t=x_0+\int_0^t\beta_sds$ for some progressive $\beta$. By (A2), the optional projection of $\beta$ w.r.t.\ $\left(\pi,\{\emptyset,\Ccal\}\times\FF^2\right)$, which we call $\bar\beta$, is well defined. As in \cite{ABZ}, one can prove that the process $M_t:=\sw_t-x_0-\int_0^t\bar\beta_s(\sw)ds$ is a $(\bmu,\FF^2)$-martingale. Indeed, taking
$0\leq s<t\leq T$ and $h_s\in L^\infty(\FF^2_s)$, we have
\begin{eqnarray*}
{\mathbb E}^{\bmu} [(M_t-M_s)h_s(\sw)] &=&  \textstyle {\mathbb E}^{\pi} [(\omega_t-\omega_s)h_s(\sw)] + {\mathbb E}^{\pi} \big[\int_{s}^{t}\big((\dot{\wideparen{\sw-\omega}})_r
-\bar\beta_r(\sw)\big)dr h_s(\sw)\big]\\
&=& \textstyle {\mathbb E}^{\pi} \big[\int_{s}^{t}{\mathbb E}^{\pi}\big[
\big((\dot{\wideparen{\sw-\omega}})_r
-\bar\beta_r(\sw)\big)|\FF^2_r\big]dr h_s(\sw)\big] \,=\, 0,
\end{eqnarray*}
where the third equality follows since $\omega$, which is a $(\bgamma,\FF^1)$-martingale, is consequently by causality a $(\pi, \FF^1\otimes \FF^2)$-martingale. Therefore $\bM$ is a $(\bmu,\FF^2)$-martingale, as claimed.

Since $\langle\bM\rangle_t=\langle\sw\rangle_t=t$ under $\bmu$, then $\bM$ is actually a $(\bmu,\FF^2)$-Brownian motion, by L\'evy's theorem.
This implies $\hat\pi:=
\cL_{\bmu}(\bM,\sw)\in\Pi_c(\bgamma,\bmu)$.
We are next going to show that the expectation in \eqref{eq:VPMKV} is smaller when considering $\hat\pi$ instead of $\pi$, i.e., when replacing $\beta=\dot{\wideparen{\sw-\omega}}$ with $\bar\beta$. Then, by taking $\Omega=\Ccal, \PP=\bmu, \FF=\FF^2, \bX=\sw$ and $\alpha=b^{-1}(\sw_t,.,\mu_t)(\bar\beta_t)$, we have a feasible tuple, which concludes the proof of the proposition.

Let us show our claim. Set $\bar{u}_t(\sw,\bmu):=b^{-1}(\sw_t,.,\mu_t)(\bar\beta_t)$ and note that, by (A2) and Jensen's inequality,
\[
f\left(\sw_t,\bar{u}_t(\sw,\bmu),\cL_\pi(\sw_t,\bar{u}_t(\sw,\bmu))\right)\leq {\mathbb E}^\pi\left[f\left(\sw_t,u_t(\omega,\sw,\mu_t),\cL_\pi(\sw_t,\bar{u}_t(\sw,\bmu))\right)|\FF^2_t\right].
\]
By taking expectation on both sides, integrating 
and using Fubini's theorem, we then get
\begin{eqnarray*}\textstyle
{\mathbb E}^\pi\left[\int_0^Tf\left(\sw_t,\bar{u}_t(\sw,\bmu),\cL_\pi(\sw_t,\bar{u}_t(\sw,\bmu))\right)dt\right]\hspace{3cm}\\ \textstyle
\leq {\mathbb E}^\pi\left[\int_0^Tf\left(\sw_t,u_t(\omega,\sw,\mu_t),\cL_\pi(\sw_t,\bar{u}_t(\sw,\bmu))\right)dt\right].
\end{eqnarray*}

We now establish some ordering between measures. For any measurable function $F:\Ccal\times\Ccal\to\RR$ and sigma-field $\sigma$, set $\bar F:={\mathbb E}^\pi[F|\sigma]$, and note that for any convex function $q:\RR\to\RR$, Jensen's inequality gives
$
\int q(x)d(\cL_\pi(\bar F))(x)={\mathbb E}^\pi[q(\bar F)]]\leq {\mathbb E}^\pi[q(F)]=\int q(x)d(\cL_\pi(F))(x),
$
i.e., $\cL_\pi(\bar F)\prec_c \cL_\pi(F)$. Analogously, for any convex function $H$, we have that $\cL_\pi(H(\bar F))\prec_{cm} \cL_\pi(H(F))$. By (A1) and (A3) this implies
\begin{eqnarray*}\textstyle
{\mathbb E}^\pi\left[\int_0^Tf\left(\sw_t,u_t(\omega,\sw,\mu_t),\cL_\pi(\sw_t,\bar{u}_t(\sw,\bmu))\right)dt\right]\hspace{3cm}\\ \textstyle
\leq {\mathbb E}^\pi\left[\int_0^Tf\left(\sw_t,u_t(\omega,\sw,\mu_t),\cL_\pi(\sw_t,\bar{u}_t(\sw,\bmu))\right)dt\right],
\end{eqnarray*}
which concludes our claim, and so the proof of the proposition.
\end{proof}

\section{Measurable selection of pushforwarding maps}

The next result is obvious in dimension one. In higher dimensions it could follow easily from Brenier's theorem in optimal transport, under assumptions relating to the finiteness of second-moments. We do not assume this, an therefore we need to be more careful. For the meaning of concepts such as $c$-cyclical monotonicity, we refer to \cite{V2}.

\begin{lem}
\label{lem selection applica}
 Let $Q$ be a probability measure on $\mathbb{R}^r\times \mathbb{R}^\ell$ and denote by $q$ the (joint) distribution of the first $r$-coordinates of $Q$.
 Then there exists a Borel measurable function $F:\mathbb{R}^r\times [0,1]^\ell \to \mathbb{R}^\ell$
 such that
 $(I,F)(q\otimes L ) = Q$,
 where $L$ denotes the restriction of $\ell$-dimensional Lebesgue measure to $[0,1]^\ell$, and $I:\mathbb{R}^r\times [0,1]^\ell \to \mathbb{R}^r$ is given by $I(x,y)=x$.
\end{lem}

\begin{proof}
Let $\mathbb{R}^r \ni x\mapsto Q^x$ be a regular disintegration of $Q$ with respect to the first $r$ coordinates. Consider the Borel function
$x\mapsto(L,Q^x)\in(\mathcal{P}(\mathbb{R}^\ell))^2$.
All assumptions of \cite[Corollary 10.44]{V2} are satisfied. Thus we have ($q(dx)$-almost surely) the existence of a unique Borel mapping $F_x(\cdot):\mathbb{R}^\ell\to \mathbb{R}^\ell$ such that $F_x(L)=Q^x$ and such that its graph is cyclically monotone (i.e.\ $c$-cyclically monotone for $c=\| \cdot\|^2$). 
By Lemma \ref{lem selection abstract} below, there exists a Borel function $F:\mathbb{R}^r\times \mathbb{R}^\ell \to \mathbb{R}^\ell$ such that
$F(x,L) = Q^x ,\, q(dx)$-a.s. We finally verify that $F(q\otimes L ) = Q$, which concludes the proof:
\begin{align*}\textstyle
\int\int (h\circ (I,F)) dq\otimes dL &\textstyle = \int \left(\int h(x, F(x,y))L(dy)\right ) q(dx) \, = \int \left (\int h(x,y)F(x,L)(dy)\right ) q(dx) \\ &\textstyle = \int \left(\int h(x,y)Q^x(dy)\right ) q(dx) \, = \int hdQ.
\end{align*}
\end{proof}

\begin{lem}
\label{lem selection abstract}
Let $(E,\Sigma,m)$ be a $\sigma$-finite measure space and consider a measurable function $E\ni \lambda\mapsto (\mu_\lambda,\nu_\lambda)\in \mathcal{P}(\mathbb{R}^\ell)\times \mathcal{P}(\mathbb{R}^\ell)$. We are also given $c:\mathbb{R}^\ell\times \mathbb{R}^\ell\to \mathbb{R}$ continuous and bounded from below. Assume that for $m$-a.e. $\lambda$, there exists a unique mapping $F_\lambda:\mathbb{R}^\ell\to \mathbb{R}^\ell$ satisfying: $F_\lambda$ is Borel measurable with $F_\lambda(\mu_\lambda)=\nu_\lambda$, and the graph of $F_\lambda$ is $c$-cyclically monotone. Then there exists $F:E\times \mathbb{R}^\ell \to \mathbb{R}^\ell$ measurable such that $m(d\lambda)$-a.s:
$F(\lambda,y)=F_\lambda(y),\, \mu_\lambda(dy)$-a.s.
\end{lem}

\begin{proof}
Let $\tilde{\Pi}(\mu,\nu):=\{\pi \in \Pi(\mu,\nu): \, \text{supp($\pi$) is $c$-cyclically monotone}\}$. We first remark that the set-valued map $(\mu,\nu)\mapsto \tilde{\Pi}(\mu,\nu)$ is measurable. To wit, $\tilde{\Pi}(\mu,\nu)$ is closed and the pre-image of closed sets by $\tilde{\Pi}(\cdot,\cdot)$ are closed. The argument for the first fact is contained in the proof of Theorem 5.20 in \cite[p.77]{V2}. As for the second fact, let $\Sigma\subset \mathcal{P}(\mathbb{R}^\ell\times \mathbb{R}^\ell)$ be closed, and $(\mu_n,\nu_n)\rightarrow (\mu,\nu)$ with $(\mu_n,\nu_n)\in \tilde{\Pi}^{-1}(\Sigma) $. The latter means that there exists $\pi_n\in \Pi(\mu,\nu)\cap \Sigma $ with supp$(\pi_n)$ being $c$-cyclically monotone. By Prokhorov, up to selection of a subsequence, we may assume that $\pi_n\rightarrow \pi \in \Pi(\mu,\nu)\cap \Sigma$, and again argumenting as in the proof of Theorem 5.20 in \cite{V2} we also get that $\pi$ has $c$-cyclically monotone support. This implies $(\mu,\nu)\in \tilde{\Pi}^{-1}(\Sigma) $, and all in all we get the measurability of $\tilde{\Pi}(\cdot,\cdot)$. We also remark that $\tilde{\Pi}(\mu,\nu)\neq \emptyset$, by the argument in the first paragraph of the proof of Theorem 10.42 in \cite[p.251]{V2}. We now closely follow the arguments in the proof of \cite[Theorem 1.1]{FGM}.
 First remark that the set-valued mapping
 $$(\mu,\nu)\mapsto \Phi(\mu,\nu):=\overline{\cup_{\pi\in \tilde{\Pi}(\mu,\nu)}\text{supp($\pi$)}}\subset \mathbb{R}^\ell\times \mathbb{R}^\ell$$
 is measurable. This easily follows, similarly to \cite[Theorem 2.1]{FGM}, by the measurability of $(\mu,\nu)\mapsto \tilde{\Pi}(\mu,\nu)$.  Now \cite[Corollary 2.3]{FGM} is valid for our $\Phi$ without any changes. Finally, the proof of \cite[Theorem 1.1]{FGM} can be fully translated in our terms.
 
\end{proof}

We provide the missing argument for Step 3 in the proof of Theorem \ref{thm approx}, and specifically Remark \ref{rem SDE}. We use the notation adopted in that part of the article.

\begin{lem}\label{lem SDE}
Given $\bnu\in\Pcal(\mathbb{R}^{2d})$ and $\pi\in\Pi_c^0(\bgamma_0,\bnu)$, there exist $\bmu\in\Pcal(\Ccal)$ and $P\in\Pi_c(\bgamma,\bmu)$ such that 
$\Lcal_P(\omega_0,\omega_1,\sw_0,\sw_1)=\pi$. This measure $P$ is the joint law of the unique weak solution of an SDE of the form
$dX_t = \beta_t dt + dW_t,$
namely $P=\Lcal(\bW,\bX)$.
\end{lem}

\begin{proof}
Recall that $\gamma_0(dz_0,dz_1)=\delta_0(dz_0)\Ncal(dz_1)$ where $\Ncal$ is the standard Gaussian in $\mathbb{R}^d$. We consider a probability space supporting a random variable $U$ uniformly distributed in $[0,1]^d$, a random variable $X_0$  distributed according to the first marginal of $\bnu$, and a standard Brownian motion $\bW$, such that $U,X_0,\bW$ are independent.
We first observe that, by Lemma \ref{lem selection applica}, there exists a Borel function $\Psi:\R^d\times \R^d\times \R^d\to\R^d$ such that
$$(\,0,W_1,X_0,\Psi(U,W_1,X_0)\,)\sim \pi.$$
Second, we define the following SDE, with initial condition $X_0$:
$$\textstyle dX_t = \frac{\Psi(U,W_t,X_0)-X_t}{1-t}dt + dW_t.$$
Note that there is at most one solution to this SDE in every interval $[0,T]$ with $T<1$, by the theory of Lipschitz SDEs with random coefficients. This proves that the solution is unique in $[0,1)$. 
Third, we observe that a solution of the above SDE is given by
$$\textstyle X_t = X_0(1-t) + (1-t)\int_0^t \frac{\Psi(U,W_s,X_0)}{(1-s)^2}ds + (1-t)\int_0^t \frac{1}{1-s}dW_s\,,$$
and therefore this is the unique solution in $[0,1)$. Finally, we observe that sending $t\to 1$ (by L'Hopital rule) we have $X_1:=X_{1-}=\Psi(U,W_1,X_0)$.
We now observe that $\Lcal(W_0,W_1,X_0,X_1)=\pi$ as desired, and notice that $P:=\Lcal(\bW,\bX)$ is causal, since $\bX$ is adapted to the filtration $\Gcal_t:=\{(U,X_0,W_s):s\leq t\}$ and $\bW$ is a $\Gcal$-Brownian motion.   
\end{proof}

\bibliographystyle{plain}
\bibliography{cite}

\end{document}